\documentclass{amsart}

\usepackage{graphicx,amsmath,amssymb,amsthm,enumerate}
\usepackage[colorlinks=true,linkcolor=black,citecolor=black,backref=page]{hyperref}
\usepackage{tikz}
\usetikzlibrary{arrows}
\newtheorem{thm}{Theorem}[section]
\newtheorem{fact}[thm]{Fact}
\newtheorem{cor}[thm]{Corollary}
\newtheorem{lem}[thm]{Lemma}
\newtheorem{prop}[thm]{Proposition}

\theoremstyle{definition}
\newtheorem{convn}[thm]{Convention}
\newtheorem{defn}[thm]{Definition}
\newtheorem{exm}[thm]{Example}
\newtheorem{rem}[thm]{Remark}
\newtheorem{quest}[thm]{Question}

\numberwithin{equation}{section}

\DeclareMathOperator{\N}{\mathbb {N}}
\DeclareMathOperator{\Z}{\mathbb {Z}}
\DeclareMathOperator{\Q}{\mathbb {Q}}
\DeclareMathOperator{\R}{\mathbb {R}}

\DeclareMathOperator{\cg}{cor}

\DeclareMathOperator{\conv}{conv}
\DeclareMathOperator{\depth}{depth}

\DeclareMathOperator{\gr}{gr}
\DeclareMathOperator{\lcm}{lcm}
\DeclareMathOperator{\lind}{ld}
\DeclareMathOperator{\Min}{Min}
\DeclareMathOperator{\pd}{pd}
\DeclareMathOperator{\rank}{rank}
\DeclareMathOperator{\reg}{reg}
\DeclareMathOperator{\supp}{supp}
\DeclareMathOperator{\smp}{\mathcal{SP}}
\DeclareMathOperator{\Tor}{Tor}

\newcommand{\mm}{\mathfrak m}
\newcommand{\nn}{\mathfrak n}
\newcommand{\pp}{\mathfrak p}
\def\alb {\boldsymbol {\alpha}}
\def\btb {\boldsymbol {\beta}}
\def\gmb {\boldsymbol {\gamma}}
\newcommand{\Fcc}{\mathcal F}
\newcommand{\Hcc}{\mathcal H}
\newcommand{\Pbb}{\mathbb P}
\newcommand{\Rcc}{\mathcal R}
\def \Gcc {\mathcal G}

\def\a {\mathbf a}
\def\e {\mathbf e}
\def\u {\mathbf u}
\def\v {\mathbf v}
\def\x {\mathbf x}

\begin{document}

\title[Regularity and Koszul property]{Regularity and Koszul property of symbolic powers of monomial ideals}

\author[L.X. Dung]{Le Xuan Dung}
\address{Faculty of Natural Sciences, Hong Duc University
No. 565 Quang Trung, Dong Ve, Thanh Hoa, Vietnam}
\email{lxdung27@gmail.com}

\author[T.T. Hien]{Truong Thi Hien}
\address{Faculty of Natural Sciences, Hong Duc University
No. 565 Quang Trung, Dong Ve, Thanh Hoa, Vietnam}
\email{hientruong86@gmail.com}

\author[H.D. Nguyen]{Hop D. Nguyen}
\address{Institute of Mathematics, VAST, 18 Hoang Quoc Viet, Hanoi, Viet Nam}
\email{ngdhop@gmail.com}

\author[T.N. Trung]{ Tran Nam Trung}
\address{Institute of Mathematics, VAST, 18 Hoang Quoc Viet, Hanoi, Viet Nam, and TIMAS, Thang Long University, Ha Noi, Vietnam.}
\email{tntrung@math.ac.vn}

\date{}

\begin{abstract}  Let $I$ be a homogeneous ideal in a polynomial ring over a field. Let $I^{(n)}$ be the $n$-th symbolic power of $I$. Motivated by results about ordinary powers of $I$, we study the asymptotic behavior of the regularity function $\reg (I^{(n)})$ and the maximal generating degree function $\omega(I^{(n)})$, when $I$ is a monomial ideal. It is known that both functions are eventually quasi-linear. We show that, in addition, the sequences $\{\reg I^{(n)}/n\}_n$ and $\{\omega(I^{(n)})/n\}_n$ converge to the same limit, which can be described combinatorially. We construct an example of an equidimensional, height two squarefree monomial ideal $I$ for which $\omega(I^{(n)})$ and  $\reg (I^{(n)})$ are not eventually linear functions. For the last goal, we introduce a new method for establishing the componentwise linearity of ideals. This method allows us to identify a new class of monomial ideals whose symbolic powers are componentwise linear. 
\end{abstract}

\maketitle
\section{Introduction}

Let $R = k[x_1,\ldots, x_r]$ be a polynomial ring over a field $k$. In this paper we investigate the maximal generating degree and the regularity of symbolic powers of  monomial ideals in $R$. Let $I$ be a homogeneous ideal of $R$. Then the $n$-th \emph{symbolic power} of $I$ is defined by
$$I^{(n)}= \bigcap_{\pp\in \Min(I)} I^nR_\pp\cap R,$$
where $\Min(I)$ is as usual the set of minimal associated prime ideals of $I$.

Symbolic powers were studied by many authors. While sharing some similar features with ordinary powers, the symbolic powers are usually much harder to deal with. One difficulty lies in the fact that the symbolic Rees algebra, defined as
\[
\Rcc_s(I)=R\oplus I^{(1)} \oplus I^{(2)} \oplus \cdots,
\]
is not noetherian in general. Examples of non-noetherian symbolic Rees algebras were discovered by Roberts \cite{Rob} and simpler examples were provided by Goto-Nishida-Watanabe \cite{GNW}. 

Denote by $\reg(I)$ and $\omega(I)$ to be the regularity of $I$ and the maximal degree of the minimal homogeneous generators $I$, respectively. By celebrated results by Cutkosky-Herzog-Trung \cite{CHT} and Kodiyalam \cite{K}, we know that $\reg I^n$ and $\omega(I^n)$ are eventually linear functions with the same leading coefficient. In particular, there exist the limits
\[
\lim_{n\to \infty} \frac{\reg I^n}{n}=\lim_{n\to \infty} \frac{\omega(I^n)}{n}
\]
and the common limit is an integer. On the other hand, by \cite[Proposition 7]{CTV}, when $I$ defines $2r+1$ points on a rational normal curve in $\Pbb^r$, where $r\ge 2$, then for all $n\ge 1$,
$$
\reg I^{(n)}=2n+1+\left\lfloor \frac{n-2}{r} \right \rfloor.
$$
Hence the function $\reg I^{(n)}$ is not eventually linear in general. Cutkosky \cite{Cut} could even construct a smooth curve in $\Pbb^3$ whose homogeneous defining ideal $I$ has the property that $\lim_{n\to \infty} \reg I^{(n)}/n$ is an irrational number. Another peculiar example is given in  \cite[Example 4.4]{CHT}: given any prime number $p\equiv 2$ modulo 3, there exist some field $k$ of characteristic $p$, and some collection of 17 fat points in $\Pbb^2_k$ whose defining ideal $I$ has the property that $\reg I^{(n)}$ is \emph{not eventually quasi-linear}.

While the question about eventual quasi-linear behavior of $\reg I^{(n)}$ has a negative answer in general, various basic questions remain tantalizing. For example:
\begin{enumerate}
 \item There was no known example of a homogeneous ideal $I$ in a polynomial ring for which the limit $\lim_{n\to\infty} \reg (I^{(n)})/n$ does not exist (Herzog-Hoa-Trung \cite[Question 2]{HeHoT});
 \item It remains an open question whether for every such homogeneous ideal $I$, the function $\reg I^{(n)}$ is bounded by a linear function;
 \item Even an answer for the analogue  of the last question for  $\omega(I^{(n)})$ remains unknown.
\end{enumerate}

In \cite[Theorem 4.9]{HoaTr}, it is shown that  $\lim_{n\to\infty} \reg (I^{(n)})/n$ exists if $I$ is a squarefree monomial ideal (but a description of the limit  was not provided). By \cite[Section 2]{HeHoT}, Question (2) (and hence of course (3)) has a positive answer if  either $I$ is a monomial ideal, or $\dim(R/I)\le 2$, or the singular locus of $R/I$ has dimension at most 1. The general case remains open for all of these questions.

Symbolic powers of monomial ideals are simpler than that of general ideals because symbolic Rees algebras of monomial ideals are noetherian \cite[Proposition 1]{Lyu}, \cite[Theorem 3.2]{HHT}. In the present paper, we address the following  questions for a monomial ideal $I$ of $R$. 

\begin{quest}
\label{quest_1_limits}
Does the limit $\lim\limits_{n\to\infty} \dfrac{\reg (I^{(n)})}{n}$ exist? If  it does, describe the limit in terms of $I$. The same questions for $\lim\limits_{n\to\infty} \dfrac{\omega(I^{(n)})}{n}$.
\end{quest}

\begin{quest}[Minh-T.N. Trung {\cite[Question A, part (i)]{MTr}}]
\label{quest_2_linear_behavior}
 Is the function $\reg I^{(n)}$ eventually linear if $I$ is squarefree?
\end{quest}
A motivation for Question \ref{quest_2_linear_behavior} is a result of Herzog, Hibi, Trung \cite{HHT}, that $\reg I^{(n)}$ is eventually quasi-linear. Another motivation is a recent result of Hoa et al. \cite{HKTT} on the existence of $\lim\limits_{n\to \infty}\depth I^{(n)}$ when $I$ is a squarefree monomial ideal. It is worth pointing out that Question \ref{quest_2_linear_behavior} has a negative answer for non-squarefree monomial ideals; see Example \ref{exm_non-linear}.

Extending previous result of Hoa and T.N. Trung, our first main result answers Question \ref{quest_1_limits} in the positive for \emph{both} limits (they are actually the same). We also describe explicitly the limits in terms of certain polyhedron associated to $I$. Our second main result answers the other question in the negative. In fact, a counterexample is given using equidimensional height 2 squarefree monomial ideals, in other words, \emph{cover ideals} of graphs. Interestingly, at the same time, our counterexample also gives a negative answer for the analogue of Question \ref{quest_2_linear_behavior} for the function $\omega(I^{(n)})$.

In detail, the main tool for Question \ref{quest_1_limits} comes from the theory of convex polyhedra. Assume that $I$ admits a minimal primary decomposition
$$I = Q_1\cap \cdots\cap Q_s\cap Q_{s+1}\cap \cdots \cap Q_t$$
where $Q_1,\ldots,Q_s$ are all the primary monomial ideals associated to the minimal prime ideals of $I$. We define certain polyhedron associated to $I$ as follows:  
$$
\smp(I)=NP(Q_1)\cap\cdots \cap NP(Q_s) \subset \R^r,
$$ 
where $NP(Q_i)$ is the Newton polyhedron of $Q_i$. Then $\smp(I)$ is a convex polyhedron in $\R^r$. For a vector  $\v = (v_1,\ldots,v_r)\in \R^r$, denote $|\mathbf v| = v_1+\cdots+v_r$. Let
$$\delta(I) = \max\{|\v| \mid \v \text { is a vertex of } \smp(I)\}.$$ 
Answering Question \ref{quest_1_limits}, our first two main results are:
\begin{thm}[Theorems \ref{thm_limit_d} and \ref{thm_limit_reg}]
For all monomial ideals $I$, there are equalities
$$
\lim_{n\to\infty} \dfrac{\omega(I^{(n)})}{n} = \lim_{n\to\infty} \dfrac{\reg (I^{(n)})}{n} = \delta(I).
$$
\end{thm}
While computing the regularity of the symbolic powers of $I$ it is difficult, the computation of $\delta(I)$ is fairly simple  by linear programming technique. 

It is not hard to show that $\delta(I)\ge \omega(I)$ if $I$ is a squarefree monomial ideal (Lemma \ref{lem_upper_bound_omega_via_delta}). Moreover, there are many examples in which $\delta(I)=\omega(I)$. This is the case when $I$ is a quadratic squarefree monomial ideal (hence $\omega(I)=2$), thanks to a result by Bahiano \cite{B}; see Example \ref{ex_delta_edge_ideals}. Let $G$ be an arbitrary simple graph with the vertex set $V(G)=\{1,\ldots,r\}$ and the edge set $E(G)$. Recall that the \emph{cover ideal} of $G$ is defined by
$$
J(G) = \bigcap_{\{i,j\} \in E(G)} (x_i,x_j).
$$
We also prove in Theorem \ref{thm_delta_JG_equal_taumaxG} that $\delta(J(G))=\omega(J(G))$, if $G$ is either bipartite, unmixed, or claw-free. An exact formula for $\reg J(G)^{(n)}$ remains elusive even for such graphs; see, for example, \cite{SF1,SF2}, for related work. Proposition \ref{prop_non_bipartite_delta_JG_equal_taumaxG} provides another large class of graphs for which the equality $\delta(J(G))=\omega(J(G))$ holds. Our main tool for proving Theorem \ref{thm_delta_JG_equal_taumaxG} and Proposition \ref{prop_non_bipartite_delta_JG_equal_taumaxG} is a combinatorial formula for $\delta(J(G))$ in Theorem \ref{thm_delta_JG}. It looks challenging to interpret $\delta(J(G))$ in terms of other known graph-theoretical invariants of $G$.

We next study componentwise linear ideals in the sense of Herzog and Hibi \cite{HH} which are also known as \emph{Koszul} ideals \cite{HIy}. Our main tool is the following new result on Koszul ideals, which is proved by the theory of linearity defect. 

\begin{prop}[See Theorem \ref{thm_ysplit}]
\label{prop_Koszul_criterion}
Let $R$ be a polynomial ring over $k$ with the graded maximal ideal $\mm$. Let $x$ be a non-zero linear form, $I'$ and $T$ non-zero homogeneous ideals of $R$ such that the following conditions are simultaneously satisfied:
\begin{enumerate}[\quad \rm(i)]
 \item $I'$ is Koszul;
 \item $T\subseteq \mm I'$;
 \item $x$ is a regular element with respect to $R/T$ and $\gr_\mm T$, the associated graded module of $T$ with respect to the $\mm$-adic filtration.
\end{enumerate}
Denote $I=xI'+T$. Then $I$ is Koszul if and only if $T$ is so.
\end{prop}
A common method (among a dozen of others), to establish the Koszul property of an ideal is to show that it has linear quotients. Compared with this method, the criterion of Proposition \ref{prop_Koszul_criterion} has the advantage that it does not require the knowledge of a system of generators of the ideal. It just asks for the knowledge of a decomposition which is in many cases not hard to obtain, the more so if we work with monomial ideals. Indeed, let $I$ be a monomial ideal of $R$, and $x$ one of its variables. Then we always have a decomposition $I=xI'+T$, where $I', T$ are monomial ideals, and $x$ does not divide any minimal generator of $T$. For such a decomposition, condition (iii) in Proposition \ref{prop_Koszul_criterion} is automatic. Hence given conditions (i) and (ii), we can prove the Koszulness of $I$ by passing to $T$, which lives in a smaller polynomial ring.  

Proposition \ref{prop_Koszul_criterion} is interesting in its own and has further applications, which we hope to pursue in future work. The main application of this proposition in our paper is to study the Koszulness of symbolic powers of the cover ideal $J(G)$ of a graph $G$. By using Proposition \ref{prop_Koszul_criterion}, we prove:

\begin{thm}[Theorem \ref{thm_Koszul_symbolicpower_coverideals}]
\label{thm_main_3_Koszul} 
Let $G$ be the graph obtained by adding to each vertex of a graph $H$ at least one pendant. Then all the symbolic powers of $J(G)$ are Koszul.
\end{thm}

It is worth mentioning that, via Alexander duality, this can be seen as a generalization of previous work of Villarreal \cite{Vi1} and Francisco-H\`a \cite{FH} on the Cohen-Macaulay property of graphs.

In order to give a counter-example to Question \ref{quest_2_linear_behavior}, we apply Theorem \ref{thm_main_3_Koszul} for corona graphs. 

\begin{thm}[Theorem \ref{thm_non-linear}]
\label{thm_main_non-linear}
For $m\geqslant 3$ and $s\geqslant 2$, let $G=\cg(K_m,s)$ be the graph obtained from the complete graph on $m$ vertices $K_m$ by adding exactly $s$ pendants to each of its vertex. Let $J = J(G)$. Then for all $n\geqslant 0$,
\begin{enumerate}[\quad \rm(1)]
\item $\reg (J^{(2n)})=\omega(J^{(2n)}) = m(s+1)n$;
\item $\reg (J^{(2n+1)})=\omega(J^{(2n+1)}) = m(s+1)n + m+s-1$.
\end{enumerate}
In particular, for all $n$,
$$
\reg (J^{(n)})=\omega(J^{(n)})=(m+s-1)n + (m-2)(s-1) \left\lfloor \frac{n}{2} \right \rfloor,
$$ 
which is not an eventually linear function of $n$.

\end{thm}

\begin{figure}[ht]
\centering
 \includegraphics[scale=0.7]{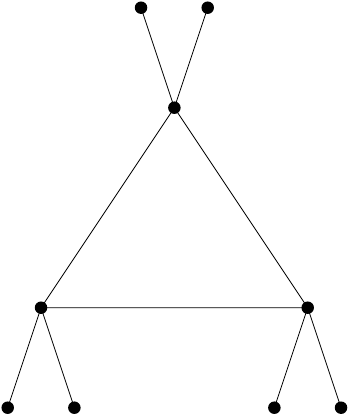}
\caption{The graph $\cg(K_3,2)$}
\label{fig_corona}
\end{figure}

Let us summarize the structure of this article. In Section \ref{sect_prelim}, we recall some necessary background. In Section \ref{sect_limit}, we prove that for any monomial ideal $I$, the limits $\lim\limits_{n\to \infty} \reg I^{(n)}/n $ and $\lim\limits_{n\to \infty} \omega(I^{(n)})/n$ exist and equal to each other. We identify them in terms of the afore-mentioned polyhedron associated to $I$. In Section \ref{sect_coverideals}, we describe structural properties of the symbolic powers of a cover ideal $J(G)$, and compute the function $\omega(J(G)^{(n)})$ in terms of the graph $G$ in certain situations. We are able to show that $\delta(J(G))=\omega(J(G))$ for graphs which are either bipartite, unmixed, or claw-free (Theorem \ref{thm_delta_JG_equal_taumaxG}). In Section \ref{sect_Koszulproperty}, we first prove the Koszulness criterion of Proposition \ref{prop_Koszul_criterion}. The main result of this section is the Koszul property of the symbolic powers for certain class of cover ideals, stated in Theorem \ref{thm_Koszul_symbolicpower_coverideals}. Combining this with results in Section \ref{sect_coverideals}, Theorem \ref{thm_main_non-linear} is deduced at the end of this section.


\section{Preliminaries}
\label{sect_prelim}
 
For standard terminology and results in commutative algebra, we refer to the book of Eisenbud \cite{Eis}. Good references for algebraic aspects of monomial ideals and simplicial complexes are the books of Herzog and Hibi \cite{HH2}, Miller and Sturmfels \cite{MS}, and Villarreal \cite{Vi}.
 
\subsection{Regularity}
Let $R$ be a standard graded algebra over a field $k$. Let $M$ be a finitely generated graded nonzero $R$-module.  Let
$$
F: \cdots \longrightarrow F_p \longrightarrow F_{p-1}\longrightarrow\cdots \longrightarrow F_1\longrightarrow F_0\longrightarrow 0
$$
be the minimal graded free resolution of $M$ over $R$. For each $i\ge 0$, $j\in \Z$, denote $\beta^R_i(M)=\rank F_i=\dim_k \Tor^R_i(k,M)$ and $\beta^R_{i,j}(M)=\dim_k \Tor^R_i(k,M)_j$. We usually omit the superscript $R$ and write simply $\beta_i(M)$ and $\beta_{i,j}(M)$ whenever this is possible. Let
$$t_i (M) = \sup\{j \mid \beta_{i,j}(M) \ne 0\}$$
where, by convention, $t_i (M) = -\infty$ if $F_i = 0$. The \emph{Castelnuovo–Mumford regularity} of $M$ measures the growth of the generating degrees of the $F_i$, $i\ge 0$. Concretely, it is defined by
$$\reg_R(M) =  \sup\{t_i(M) -i \mid  i \geqslant 0\}.$$

In the remaining of this paper, we denote by $\omega(M)$ the number $t_0(M)$. Hence $\omega(M)$ is the maximal degree of a minimal homogeneous generator of $M$. The definition of the regularity implies $$\omega(M) \leqslant \reg_R(M).$$
If $M$ is generated by elements of the same degree $d$, and $\reg_R M=d$, we say that $M$ has a \emph{linear resolution} over $R$. We also say $M$ has a $d$-linear resolution in this case.

If $R$ is a standard graded polynomial ring over $k$, it is customary to denote $\reg_R M$ simply by $\reg M$.

\subsection{Linearity defect, Koszul modules, Betti splittings}
We use the notion of linearity defect, formally introduced by Herzog and Iyengar \cite{HIy}. Let $R$ be a standard graded $k$-algebra, and $M$ a finitely generated graded $R$-module. The linearity defect of $M$ over $R$, denoted by $\lind_R M$, is defined via certain filtration of the minimal graded free resolution of $M$. For details of this construction, we refer to \cite[Section 1]{HIy}. We say $M$ is called a \emph{Koszul module} if $\lind_R M=0$. Koszul modules are those modules with a linear resolution in the terminology of \c{S}ega \cite{Se}. We say that $R$ is a \emph{Koszul algebra} if $\reg_R k=0$. As a matter of fact, $R$ is a Koszul algebra if and only if $k$ is a Koszul $R$-module \cite[Remark 1.10]{HIy}.

For each $d\in \Z$, denote by $M_{\langle d \rangle}$ the submodule of $M$ generated by homogeneous elements of degree $d$. Following Herzog and Hibi \cite{HH}, $M$ is called \emph{componentwise linear} if for all $d\in \Z$, $M_{\langle d \rangle}$ has a $d$-linear resolution. By results of R\"omer \cite[Theorem 3.2.8]{Ro} and Yanagawa \cite[Proposition 4.9]{Yan}, if $R$ is a Koszul algebra, then $M$ is Koszul if and only if $M$ is componentwise linear.  

Because of the last result and for unity of treatment, we use the terms \emph{Koszul} modules throughout, instead of \emph{componentwise linear} modules.

The following result is folklore; see for example \cite[Proposition 3.4]{AR}.
\begin{lem}
\label{lem_Koszul_gendeg}
Let $R$ be a standard graded $k$-algebra, and $M$ be a Koszul $R$-module. Then $\reg_R M=\omega(M)$.
\end{lem}
We also recall the following base change result for the linearity defect.
\begin{lem}[Nguyen and Vu {\cite[Corollary 3.2]{NgV2}}]
\label{lem_basechange_ld}
Let $R\to S$ be a flat extension of standard graded $k$-algebras. Let $I$ be a homogeneous ideal of $R$. Then 
\[
\lind_R I=\lind_S IS.
\] 
\end{lem}

Let $(R,\mm)$ be a noetherian local ring (or a standard graded $k$-algebra) and $P,I,J\neq (0)$ be proper (homogeneous) ideals of $R$ such that $P=I+J$. 
\begin{defn}
The decomposition of $P$ as $I+J$ is called a {\it Betti splitting} if for all $i\ge 0$, the following equality of Betti numbers holds: $\beta_i(P)=\beta_i(I)+\beta_i(J)+\beta_{i-1}(I\cap J)$.
\end{defn}
We have the following reformulations of Betti splittings.
\begin{lem}[{\cite[Lemma 3.5]{NgV2}}]
\label{lem_criterion_Bettisplit}
The following are equivalent:
\begin{enumerate}[\quad \rm(i)]
\item The decomposition $P=I+J$ is a Betti splitting;
\item The natural morphisms $\Tor^R(k,I\cap J) \to \Tor^R(k,I)$ and $\Tor^R(k,I\cap J) \to \Tor^R(k,J)$ are both zero;
\item The mapping cone construction for the map $I\cap J \to I\oplus J$ yields a minimal free resolution of $P$.
\end{enumerate}
\end{lem}

\subsection{Symbolic powers of monomial ideals}

Let $R=k[x_1,\ldots,x_r]$ be a standard graded polynomial ring, and $I$ a monomial ideal of $R$. Let $\Gcc(I)$ denotes the set of minimal monomial generators of $I$. In the present paper, when talking about minimal generators of a monomial ideal we mean \emph{minimal monomial generators} of it. Let
$$I = Q_1\cap \cdots\cap Q_s\cap Q_{s+1}\cap \cdots \cap Q_t$$
be a minimal primary decomposition of $I$, where $Q_i$ is a primary monomial ideal for $i=1,\ldots,t$, and $P_j=\sqrt{Q_j}$ is a minimal prime of $I$ if and only if $1\le j \le s$. (Hence $Q_{s+1},\ldots,Q_t$ are embedded primary components.) For each $i=1,\ldots, s$, the monomial ideal $Q_j$ is obtained from minimal generators of $I$ by setting $x_i = 1$ for all $i$ for which $x_i\notin P_j$, thus
\begin{equation}
\label{deg-bound} \omega(Q_j) \leqslant \omega(I), \ \text{ for } j = 1,\ldots,s.
\end{equation}

In the case of monomial ideals, we have a simple formula for the symbolic powers in terms of the minimal primary components. It is immediate from the definition of symbolic power, and the fact that monomial primary ideals must be generated by powers of certain variables and monomials involving only of those variables.
\begin{fact}
\label{fact_symbolic_power_monomial}
With notation as above, for all $n\ge 1$, there is an equality
$$I^{(n)}=Q_1^n \cap Q_2^n \cap \cdots \cap Q_s^n.$$
\end{fact}
A function $f\colon \N\to \N\cup\{-\infty\}$ is called \emph{quasi-linear} if there exist a positive integer $N$ and rational numbers $a_i\in\Q$ and $b_i\in \Q\cup\{-\infty\}$, for $i=0,\ldots,N-1$, such that 
$$f(n)=a_in+b_i, \text{ for all } n\in\N \text{ with } n\equiv i \pmod{N}.$$
In this case, the smallest such number $N$ is called the period of $f$. 

Assume that $f$ is not identically $-\infty$. Then $\lim\limits_{n\to\infty} \dfrac{f(n)}{n}$ exists if and only if  $a_0=\cdots=a_{N-1}$. In this case, we say that $f$ has a constant leading coefficient.

\begin{lem}
With notation as above, for every $i\geqslant 0$, $t_i(I^{(n)})$ is quasi-linear in $n$ for $n\gg 0$. In particular, $\omega(I^{(n)})$ and $\reg(I^{(n)})$ are quasi-linear in $n$ for $n\gg 0$.
\end{lem}
\begin{proof} By \cite[Theorem 3.2]{HHT}, the symbolic Rees ring $\mathcal R_s(I) = \bigoplus_{n=0}^{\infty} I^{(n)}$ is finitely generated. By the very same way as the proof of \cite[Theorem 4.3]{CHT}, we obtain  $t_i(I^{(n)})$ is quasi-linear in $n$ for $n\gg 0$.
\end{proof}

If $I$ is a monomial ideal of $R$, the minimal graded free resolution of $I$ is $\Z^r$-graded. For each $\alb\in \Z^r$, we denote by $\beta_{i,\alb}(I)$ the number $\dim_k \Tor^R_i(k,I)_{\alb}$. Clearly $\beta_{i,\alb}(I)=0$ if $\alb \notin \N^r$.

When we talk about a monomial $x^{\alb}$ of $R$, we always mean $\alb=(\alpha_1,\ldots,\alpha_r)\in \N^r$ and $x^{\alb}=x_1^{\alpha_1}\cdots x_r^{\alpha_r}$. A vector $\alb = (\alpha_1,\ldots,\alpha_r)\in \N^r$ is called squarefree if for all $i=1,\ldots,r$, $\alpha_i$ is either 0 or 1. Let $\e_1,\ldots,\e_r$ be the canonical basis of the free $\Z$-module $\Z^r$.
For any $\alb = (\alpha_1,\ldots,\alpha_r)\in \N^r$ the {\it upper Koszul simplicial complex} associated with $I$ at degree $\alb$ is defined by
$$K^{\alb}(I) = \{\text{squarefree vector } \tau \mid x^{\alb-\tau} \in I\},$$
where we use the convention $\alb - \tau = \alb - \sum_{i\in\tau} \e_i$.  The multigraded Betti numbers of $I$ can be computed as follows.

\begin{lem}\label{L00} {\rm (\cite[Theorem 1.34]{MS})} For all $i\geqslant 0$ and all $\alpha \in \N^r$, there is an equality $\beta_{i,\alb}(I) = \dim_k \widetilde{H}_{i-1}(K^{\alb}(I);k)$.
\end{lem}

Let $\overline{I}$ be the integral closure of the monomial ideal $I$. To describe $\overline{I}$ geometrically, for a subset $A$ of $R$, denote 
$$
E(A) = \{\alb\mid \alb\in \N^r \text{ and } x^{\alb} \in A\}.
$$
The Newton polyhedron of $I$ is the convex polyhedron in $\R^r$ defined by $NP(I) = \conv\{E(I)\}$. Then $\overline{I}$ is a monomial ideal  determined by (see \cite[Exercises 4.22 and 4.23]{Eis}):
\begin{equation}\label{Int-meaning}
E(\overline{I}) = NP(I) \cap \N^r.
\end{equation}

For each $n\geqslant 1$, let $$\smp_n(I) = \bigcap_{i=1}^s NP(Q_i^n),$$
and $$J_n(I) = \overline{Q_1^n}\cap \cdots\cap \overline{Q_s^n}.$$
Then from Equation (\ref{Int-meaning}) we have $E(J_n(I)) = \smp_n(I) \cap \N^r$.

We will denote $\smp_1(I)$ simply by $\smp(I)$. This is the same as the symbolic polyhedron introduced in \cite[Definition 5.3]{CEHH}, if $I$ has no embedded primes. But in general, the two notions are different, since in contrast to our definition, the definition of symbolic power in \cite{CEHH} involves all the associated primes.

For subsets $X$ and $Y$ of $\R^r$ and a positive integer $n$, we denote
\begin{align*}
nX  &= \{ny\mid y\in X\},\\
X+Y &= \{x+y: x\in X, y\in Y\}.
\end{align*}
Denote by $\R_{+}$ the set of non-negative real numbers. The following lemma gives us the structure of the convex polyhedron $\smp_n(I)$.

\begin{lem}\label{NP-powers} Let $\{\v_1,\ldots,\v_d\}$ be the set of vertices of $\smp(I)$. Then
$$\smp_n(I) =  n\smp(I)= n \conv \{\v_1,\ldots,\v_d\} + \R_+^r.$$
\end{lem}
\begin{proof} For each $i=1,\ldots,s$, we have $NP(Q_j^n) = n NP(Q_j)$ by \cite[Lemma 2.5]{RRV}. It follows that $\smp_n(I) =n\smp(I)$. 

For $\v\in \smp(I)$ and $\u\in\R^r_+$, one has $\v+\u \in \smp(I)$ again by \cite[Lemma 2.5]{RRV}. Combining this with \cite[Formula (28), Page 106]{S} we have 
$$\smp(I) = \conv\{\v_1,\ldots,\v_d\} + \R_{+}^r.$$ 
Thus, $\smp_n(I) = n \smp(I) = n \conv\{\v_1,\ldots,\v_d\} + \R_{+}^r$, as required.
\end{proof}

The following result was proved in \cite[Lemma 6]{Tr}.
\begin{lem}
\label{lem_Newton_polyh}
Let $Q$ be a monomial ideal of $R$. Then the Newton polyhedron $NP(Q)$ is the set of solutions of a system of
inequalities of the form
$$
\{\x \in \R^r \mid \langle \a_j, \x\rangle \ge b_j, j = 1,\ldots, q\},
$$
such that the following conditions are simultaneously satisfied:
\begin{enumerate}[\quad \rm(i)]
 \item Each hyperplane with the equation $\langle \a_j , \x\rangle = b_j$ defines a facet of $NP(Q)$, which
contains $s_j$ affinely independent points of $E(\Gcc(Q))$ and is parallel to $r -s_j$ vectors of the
canonical basis. In this case $s_j$ is the number of non-zero coordinates of $\a_j$.
\item $\mathbf{0} \neq \a_j \in \N^r, b_j \in \N$ for all $j = 1,\ldots, q$.
\item If we write $\a_j = (a_{j,1} ,\ldots , a_{j,r})$, then $a_{j,i}\le s_j\omega(Q)^{s_j-1}$ for all $i = 1,\ldots, r$.
\end{enumerate}
\end{lem}

Using this, we can give information about facets of $\smp_n(I)$, which will be useful to bound from below the maximal generating degree of $I^{(n)}$ by some linear function of $n$.

\begin{lem}
\label{lem_facets} 
The polyhedron $\smp(I)$ is the solutions in $\R^r$ of a system of linear inequalities of the form
$$\{\x\in\R^r \mid \left<\a_j,\x\right> \geqslant b_j, \ j=1,2,\ldots, q\},$$
where for each $j$, the following conditions are fulfilled:
\begin{enumerate}[\quad \rm(i)]
 \item $\mathbf{0} \neq \a_j\in \N^r$, $b_j\in \N$;
 \item $|\a_j| \leqslant r^2\omega(I)^{r-1}$;
\item The equation $\left<\a_j,\x\right> = b_j$ defines a facet of $\smp(I)$.
\end{enumerate}
\end{lem}
\begin{proof} Note that $\smp(I)$ is the solution in $\R^r$ of the system of all linear inequalities that arise from those inequalities defining $NP(Q_j)$ where $j=1,\ldots,s$. Now combining Lemma \ref{lem_Newton_polyh} with the fact that $\omega(Q_j) \leqslant \omega(I)$ (Inequality \eqref{deg-bound}), the lemma follows.
\end{proof}

Let $\Delta$ be a simplicial complex on $\{1, \ldots , r\}$. For a subset $F =\{i_1,\ldots,i_j\}$ of $\{1, \ldots , r\}$, set $x^F = x_{i_1}\cdots x_{i_j}$ and $P_F=(x_i: i\notin F)$. Then the \emph{Stanley-Reisner ideal} of $\Delta$ is the squarefree monomial ideal
$$I_{\Delta} = (x^G \mid G \notin  \Delta) \subseteq  R.$$

Let $\mathcal F(\Delta)$ denote the set of all facets of $\Delta$. If $\mathcal F(\Delta) = \{F_1, \ldots, F_m\}$, we write $\Delta = \left <F_1, \ldots, F_m\right >$. Then $I_{\Delta}$ admits the primary decomposition
$$
I_{\Delta} = \bigcap_{F\in \mathcal F(\Delta)} P_F.
$$
Thanks to Fact \ref{fact_symbolic_power_monomial}, for every integer $n \geqslant 1$, the $n$-th symbolic power of $I_{\Delta}$ is given by
$$I_{\Delta}^{(n)} = \bigcap_{F\in \mathcal F(\Delta)} P_F^n.$$

\subsection{Graph theory}
Let $G$ be a finite simple graph. We use the symbols $V(G)$ and $E(G)$ to denote the vertex set and the edge set of $G$, respectively. When there is no confusion, the edge $\{u,v\}$ of $G$ is written simply as $uv$. Two vertices $u$ and $v$ are \emph{adjacent} if $\{u,v\} \in E(G)$.

For a subset $S$ of $V(G)$, we define $$N_G(S) = \{v\in V(G)\setminus S \mid uv\in E(G) \text{ for some } u \in S\}$$ and $N_G[S] = S \cup N_G(S)$. When there is no confusion, we shall omit $G$ and write $N(S)$ and $N[S]$.  If $S$ consists of a single vertex $u$, denote $N_G(u)=N_G(S)$ and $N_G[u]=N_G[S]$.  Define $G[S]$ to be the induced subgraph of $G$ on $S$, and $G \setminus S$ to be the subgraph of $G$ with the vertices in $S$ and their incident edges deleted. 

The \emph{degree} of a vertex $u \in V(G)$, denoted by $\deg_G(u)$, is the number of edges incident to $u$. If $\deg_G(u) = 0$, then $u$ is called an \emph{isolated vertex}; if $\deg_G(u)=1$, then $u$ is a \emph{leaf}. An edge emanating from a leaf is called a {\it pendant}.

A \emph{vertex cover} of $G$ is a subset of $V(G)$ which meets every edge of $G$; a vertex cover is \emph{minimal} if none of its proper subsets is itself a cover. The \emph{cover ideal} of $G$ is defined by
$J(G) := (x^{\tau} \mid \tau \text{ is a minimal vertex cover of } G)$. Note that $J(G)$ has the primary decomposition
$$J(G) = \bigcap_{\{i,j\}\in E(G)} (x_i,x_j).$$

An \emph{independent set} in $G$ is a set of vertices no two of which are adjacent to each other. An independent set in $G$ is \emph{maximal} (with respect to set inclusion) if the set cannot be extended to a larger independent set. The set of all independent sets of $G$, denoted by $\Delta(G)$, is a simplicial complex, called the \emph{independence complex} of $G$. 


\section{Asymptotic maximal generating degree and regularity}
\label{sect_limit}

Let $I$ be a monomial ideal of $R=k[x_1,\ldots,x_r]$ and let
$$I = Q_1\cap \cdots\cap Q_s\cap Q_{s+1}\cap \cdots \cap Q_t$$
be a minimal primary decomposition of $I$, where $Q_1,\ldots,Q_s$ are the components associated to the minimal primes of $I$.  By Fact \ref{fact_symbolic_power_monomial} we have 
$$I^{(n)}=Q_1^n \cap Q_2^n \cap \cdots \cap Q_s^n.$$
Recall that
$$\smp_n(I) = NP(Q_1^n) \cap NP(Q_2^n) \cap \cdots\cap NP(Q_s^n) = n\smp(I),$$
and
$$J_n(I) = \overline{Q_1^n}\cap \overline{Q_2^n}\cap \cdots\cap \overline{Q_s^n}.
$$
For simplicity, we denote $J_n=J_n(I)$ in the sequel. Observe that $x^{\alb} \in J_n$ if and only if $\alb \in \smp_n(I)\cap \N^r$. We note two simple facts.

\begin{rem}\label{GEN-I} Let $J$ be a monomial ideal and  $x^{\alb}\in J$, with $\alb = (\alpha_1,\ldots,\alpha_r)\in \N^r$. Then $x^{\alb}\in \Gcc(J)$ if and only if for every $i$ with $\alpha_i\ge 1$, we have $x^{\alb-\e_i}\notin J$.
\end{rem}

\begin{lem}\label{N0} Let $J$ be a monomial ideal and $x^{\alb}\in J$. For $i=1,\ldots,r$, let $m_i\in \N$ be an integer such that $x^{\alb-m_i\e_i}\notin J$ if $\alpha_i \geqslant m_i$. Then there are integers $0\leqslant n_i \leqslant m_i-1$ such that $x^{\alb -(n_1\e_1+\cdots+n_r\e_r)}\in \Gcc(J)$.
\end{lem}
\begin{proof} 
Just choose $0\leqslant n_i < m_i$ for $i=1,\ldots,r$ such that
$$x^{\alb -(n_1\e_1+\cdots+n_r\e_r)} \in J$$
and  $n_1+\cdots+n_r$ is as large as possible. 
\end{proof}

The first main result of this paper is

\begin{thm} \label{thm_limit_d} 
There is an equality $\lim_{n\to\infty} \dfrac{\omega(I^{(n)})}{n} = \delta(I)$.
\end{thm}
In fact, setting $\rho=r^2\omega(I)^{r-1}$, we will prove that for all $n\geqslant  1$, the following bounds for $\omega(I^{(n)})$ hold:
\begin{equation}
\label{eq_bound_omega(In)}
 \delta(I)n - r\rho(1+s(r-1)\omega(I)) \leqslant \omega(I^{(n)}) \leqslant \delta(I)n + r +r(r-1)\omega(I).
\end{equation}
This clearly implies the conclusion of Theorem \ref{thm_limit_d}.

For the upper bound, we need the following auxiliary statements. 

\begin{lem} 
\label{lem_fromI(n)toJn}  
Let $x^{\alb}\in I^{(n)}$ be a monomial. Assume that for some $1\le i\le r$, we have $x^{\alpha-\e_i}\notin I^{(n)}$.  Denote $m = (r-1)\omega(I)+1$. If $\alpha_i \geqslant m$, then $x^{\alb -m\e_i}\notin J_n$.
\end{lem}
\begin{proof} Since $x^{\alpha-\e_i}\notin I^{(n)}$, $x^{\alpha-\e_i}\notin Q_j^n$ for some $1\leqslant j\leqslant s$. By \cite[Theorem 7.58]{V}, we have $\overline{Q_j^n} = Q_j^{n-p}\overline{Q_j^p}$ for some $0\leqslant p \leqslant r-1$. 

Since $x^{\alb}\in Q_j^n$ and $x^{\alb-\e_i}\notin Q_j^n$, it follows that $x_i$ divides some generator of $Q_j^{n}$. As the monomial ideal $Q_j$ is primary, $x_i^{\omega(Q_j)}\in Q_j$. In particular, $x_i^{\omega(I)}\in Q_j$ because $\omega(Q_j) \leqslant \omega(I)$.

We now assume on the contrary that $x^{\alb-m\e_i}\in J_n$. Then $x^{\alb-m\e_i}\in \overline{Q_j^n}$. Since $\overline{Q_j^n} = Q_j^{n-p}\overline{Q_j^p}$, there are two monomials $m_1\in Q_j^{n-p}$ and $m_2\in \overline{Q_j^p}$ such that $x^{\alb-m\e_i} = m_1m_2$. It follows that $x^{\alb-\e_i} = (m_1x_i^{m-1}) m_2$. Observe that $x_i^{m-1}\in Q_j^{r-1}$ as $m-1 = (r-1)\omega(I)$. Thus $x^{\alb-\e_i} = (m_1x_i^{m-1}) m_2\in Q_j^{n-p}Q_j^{r-1} \subseteq Q_j^n$, a contradiction. The lemma follows.
\end{proof}

\begin{lem}
\label{lem_dJn}
There is an inequality $\omega(J_n) < \delta(I)n+r$.
\end{lem}
\begin{proof} Let $x^{\alb}\in \Gcc(J_n)$, $\v_1,\ldots,\v_d$ be all the vertices of $\smp(I)$. By Lemma \ref{NP-powers}, we can represent $\alb$ as
$$\alb = n(\lambda_1 \v_1+\cdots+\lambda_d\v_d)+\u$$
where $\lambda_i\geqslant 0$, $\lambda_1+\cdots+\lambda_d = 1$, and $\u=(u_1,\ldots,u_r) \in \R_{+}^r$.

Since $x^{\alb}$ is a minimal generator of $J_n$, necessarily $u_i < 1$ for every $i$. Therefore, $$|\alb| \leqslant \delta(I)n+(u_1+\cdots +u_r) < \delta(I)n + r.$$
It follows that $\omega(J_n) < \delta(I)n+r$, as required.
\end{proof}

Now we are ready for the 

\begin{proof}[Proof of the inequality on the right of \eqref{eq_bound_omega(In)}]

Let $x^{\alb}$ be a minimal generator of $I^{(n)}$. By Remark \ref{GEN-I} we have $x^{\alb-\e_i}\notin I^{(n)}$ for each $i=1,\ldots,r$, whenever $\alpha_i\geqslant 1$.  For $1\le i\le r$, set $ m_i =  (r-1)\omega(I)+1$. By Lemma \ref{lem_fromI(n)toJn},  $x^{\alb-m_i\e_i} \notin J_n$ if $\alpha_i \geqslant (r-1)\omega(I)+1$. 

By Lemma \ref{N0}, there are integers $0\leqslant n_i \leqslant (r-1)\omega(I)$ such that the monomial $x^{\alb-(n_1\e_1+\cdots+n_r\e_r)}$ is a minimal generator of $J_n$. Thus
$$
\omega(J_n) \geqslant |\alb|-(n_1+\cdots+n_r) \geqslant |\alb| - r(r-1)\omega(I),
$$
and hence $|\alb| \leqslant \omega(J_n) +r(r-1)\omega(I)$. It follows that $\omega(I^{(n)}) \leqslant \omega(J_n) + r(r-1)\omega(I)$. Together with Lemma \ref{lem_dJn}, we obtain
\[
\omega(I^{(n)}) \leqslant \delta(I)n + r +r(r-1)\omega(I).
\]
This is the desired inequality.
\end{proof}
For the remaining inequality in \eqref{eq_bound_omega(In)}, we will make some use of Lemma \ref{lem_facets}.
\begin{proof}[Proof of the inequality on the left of \eqref{eq_bound_omega(In)}]

Let  $\v = (v_1,\ldots,v_r)$ be a vertex of the polyhedron $\smp(I)$ such that $\delta(I) = |\v|$. Let $\alb = (\alpha_1,\ldots,\alpha_r)\in\N^r$ where $\alpha_i = \lceil nv_i \rceil$. Because $n\v$ is a vertex of $\smp_n(I)$, $x^{\alb} \in J_n$.

For each $i=1,\ldots,s$, we have $x^{\alb}\in \overline{Q_i^n} = Q_i^{n-(r-1)}\overline{Q_i^{r-1}}$ by \cite[Theorem 7.58]{V}, so we can write
$$
x^{\alb} = m_{i,1}m_{i,2}m_{i,3}
$$
where $m_{i,1}\in Q_i$, $m_{i,2} \in Q_i^{n-r}$ and $m_{i,3}\in \overline{Q_i^{r-1}}$. Let $f_i = m_{i,1}^{r-1}$ so that $\deg(f_i) \leqslant (r-1)\omega(Q_i)\leqslant (r-1)\omega(I)$. We have $x^{\alb}f_i=(m_{i,1}^rm_{i,2})m_{i,3} \in Q_i^n$.

Let $x^{\btb} = f_1\cdots f_s$ and $x^{\gmb} = x^{\alb}x^{\btb}$. Then $x^{\gmb}\in Q_i^n$ for all $i$, consequently $x^{\alb} x^{\btb}\in I^{(n)}$. Moreover, $\gamma_i = 0$ if and only if $\alpha_i =0$, if and only if $v_i = 0$. Note that $|\btb| =\deg(f_1)+\cdots+\deg(f_s) \leqslant s (r-1)\omega(I)$.

By Lemma \ref{lem_facets}, the convex polyhedron $\smp(I)$ is the solutions in $\R^r$ of a system of linear inequalities of the form
$$\{\x\in\R^r \mid \left<\a_j,\x\right> \geqslant b_j, \ j=1,2,\ldots, q\},$$
such that: 
\begin{enumerate}
 \item each equation $\left<\a_j,\x\right> = b_j$ defines a facets of $\smp(I)$,
 \item  $\a_j\in \N^r$, $b_j\in \N$, and,
 \item  $|\a_j| \leqslant r^2\omega(I)^{r-1}$ for any $j$.
\end{enumerate}

Let $\rho = r^2\omega(I)^{r-1}$ so that $|\a_j| \leqslant \rho$ for every $j=1,\ldots,q$.

Since $\v$ is a vertex of $\smp(I)$, by \cite[Formula 23 in Page 104]{S}, we may assume that $\v$ is the unique solution of the following system
$$
\left\{\x\in\R^r \mid \left<\a_i,\x\right> = b_i, i=1,\ldots,r\right\}.
$$

For an index $i$ with $\gamma_i\geqslant 1$, since the last system has a unique solution, we deduce that $\a_{j,i} \ne 0$ for some $1\leqslant j\leqslant r$. For simplicity, we denote $\a = \a_{j} = (a_1,\ldots,a_r)$ so that $a_i\geqslant 1$. 

Let $m = \rho(1+s(r-1)\omega(I)) +1$. If $\gamma_i\geqslant m$, we have
\begin{align*}
\left<\a,\gmb-m\e_i\right>&=\left<\a,\alb\right>+\left<\a,\btb\right>-a_im \leqslant \left<\a,n\v+\e_1+\cdots+\e_r\right>+\left<\a,\btb\right>-a_im\\
&=\left<\a,n\v\right>+|\a|+\left<\a,\btb\right>-a_im = nb_j + |\a|+\left<\a,\btb\right>-a_im \\
& \leqslant nb_j +|\a| +|\a||\btb|-m < nb_j
\end{align*}
since $m = \rho(1+s(r-1)\omega(I)) +1 > |\a| +|\a||\btb|$. Consequently, $x^{\gmb - m\e_i} \notin J_n$, and hence $x^{\gmb - m\e_i} \notin I^{(n)}$.

By Lemma \ref{N0}, there are non-negative integers $n_i\leqslant \rho(1+s(r-1)\omega(I))$ for $i=1,\ldots,r$ such that
$x^{\gmb -(n_1\e_1+\cdots+n_r\e_r)}$ is  a minimal generator of $I^{(n)}$. Therefore
\begin{align*}
\omega(I^{(n)}) &\geqslant |\gmb| - (n_1+\cdots+n_r) \geqslant |\alb|+|\btb| - r\rho(1+s(r-1)\omega(I))\\
&\geqslant |\alb|-r\rho(1+s(r-1)\omega(I))\geqslant |n\v|-r\rho(1+s(r-1)\omega(I))\\
&=\delta(I)n - r\rho(1+s(r-1)\omega(I)).
\end{align*}
This finishes the proof of \eqref{eq_bound_omega(In)} and hence that of Theorem \ref{thm_limit_d}.
\end{proof}

The second main result of this paper is

\begin{thm}
\label{thm_limit_reg} 
There is an equality
$\lim_{n\to\infty} \dfrac{\reg(I^{(n)})}{n} = \delta(I)$.
\end{thm}
Recall that for any finitely generated graded $R$-module $M$, and for any $i\ge 0$, we have the notation 
$$
t_i(M)=\sup\{j: \Tor^R_i(k,M)_j\neq 0\}.
$$
From Theorem \ref{thm_limit_d} and the fact that $\omega(M)\le \reg M$, we see that Theorem \ref{thm_limit_reg} will follow from a suitable linear upper bound for $\reg I^{(n)}$. This is accomplished by
\begin{lem}\label{L1} 
For all $i\geqslant 0$, there is an inequality 
$$
t_i(I^{(n)}) \leqslant \delta(I)n+2r+r(r^2\omega(I)^{r-1}+(r-1)\omega(I)).
$$
\end{lem}
\begin{proof} By Lemma \ref{lem_facets}, the convex polyhedron $\smp(I)$ is the solutions in $\R^r$ of a system of linear inequalities of the form
$$\{\x\in\R^r \mid \left<\a_j,\x\right> \geqslant b_j, \ j=1,2,\ldots, q\},$$
where for each $j$, the equation $\left<\a_j,\x\right> = b_j$ defines a facets of $\smp(I)$, $\a_j\in \N^r$, $b_j\in \N$, and $|\a_j| \leqslant r^2\omega(I)^{r-1}$. 

Let $\rho= r^2\omega(I)^{r-1}$ so that $|\a_j| \leqslant \rho$ for every $j$.

Take $\alb = (\alpha_1,\ldots,\alpha_r)\in \N^r$ such that $\beta_{i,\alb}(I_{\Delta}^{(n)}) \ne 0$.  Since $\beta_{i,\alb}(I_{\Delta}^{(n)}) = \dim_k \widetilde{H}_{i-1}(K^{\alb}(I_{\Delta}^{(n)});k)\ne 0$ by Lemma \ref{L00}, we have $K^{\alb}(I_{\Delta}^{(n)})$ is not a cone. Hence, for each $j=1,\ldots, r$, we have $j \notin \tau$ for some $\tau\in \mathcal F(K^{\alb}(I_{\Delta}^{(n)}))$. 

Since $\tau \cup\{j\} \notin K^{\alb}(I_{\Delta}^{(n)})$, we have $x^{\alb-\tau-\e_j}\notin I^{(n)}$. Let $m = (r-1)\omega(I)+1$. 

\textbf{Claim}: If $\alpha_j\geqslant \rho+m$, then $x^{\alb-(\rho+m)\e_j}\notin J_n$.

Indeed, by Lemma \ref{lem_fromI(n)toJn}, $x^{\alpha-\tau-m\e_j}\notin J_n$. Therefore, $\left <\a_i,\alb-\tau-m\e_j\right> < nb_i$ for some $1\leqslant i\leqslant q$. Since $x^{\alb-\tau}\in I^{(n)} \subseteq J_n$, we have $\left <\a_i,\alb-\tau\right> \geqslant nb_i$. It follows that $\a_{i,j}\geqslant 1$. Thus
\begin{align*}
\left <\a_i,\alb-(\rho+m)\e_j\right>&=\left <\a_i,\alb-\tau-m\e_j\right> +\left <\a_i,\tau-\rho\e_j\right> < nb_j +\langle \a_i,\tau-\rho\e_j\rangle\\
&=nb_j +\langle \a_i,\tau \rangle -\langle \a_i,\rho \e_j \rangle = nb_j +\langle \a_i,\tau \rangle -\a_{i,j} \rho\\
&\leqslant nb_j+\langle \a_i,\tau \rangle-\rho \leqslant nb_j.
\end{align*}
The last inequality holds since $\rho \geqslant |\a_i|\geqslant \langle \a_i,\tau \rangle$. Consequently, $x^{\alb-(\rho+m)\e_j}\notin J_n$, as desired.

By Lemma \ref{lem_fromI(n)toJn}, there are integers $0\leqslant n_i \leqslant \rho+m-1$ for $i=1,\ldots,r$, for which
$$x^{\alb-\tau -n_1\e_1-\cdots-n_r\e_r} \in \Gcc(J_n).
$$
It follows that
$$\omega(J_n) \geqslant |\alb| -|\tau| - (n_1+\cdots+n_r) \geqslant |\alb| -r-r(\rho+m-1),$$
and hence 
$$|\alb| \leqslant \omega(J_n)+r+r(\rho+m-1) = \omega(J_n) +r+r(r^2\omega(I)^{r-1}+(r-1)\omega(I)).$$ 
Together with Lemma \ref{lem_dJn}, this yields
\begin{align*}
 t_i(I^{(n)}) &\leqslant \omega(J_n)+r+r\left(r^2\omega(I)^{r-1}+(r-1)\omega(I)\right)\\
              &\leqslant \delta(I)n+2r+r\left(r^2\omega(I)^{r-1}+(r-1)\omega(I)\right),
\end{align*}
and the proof is complete.
\end{proof}

\begin{proof}[Proof of Theorem \ref{thm_limit_reg}]
By Lemma \ref{L1} we have
$$
\reg I^{(n)} =\max\{t_i(I^{(n)})-i \mid i\geqslant 0\} \leqslant \delta(I)n+2r+r(r^2\omega(I)^{r-1}+(r-1)\omega(I)).
$$

On the other hand,  by the proof of Theorem \ref{thm_limit_d} (more precisely \eqref{eq_bound_omega(In)}), there exists $c\in \R$ such that $$\omega(I^{(n)}) \geqslant \delta(I)n+c \, \text{ for all } n\geqslant 1.$$
In particular, $\reg I^{(n)} \geqslant \omega(I^{(n)}) \geqslant \delta(I)n+c$ for all $n\geqslant 1$. Thus,
$$\delta(I)n+c \leqslant \reg I^{(n)} \leqslant \delta(I)n+2r+r(r^2\omega(I)^{r-1}+(r-1)\omega(I))$$
for all $n\geqslant 1$. It follows that
$$\lim_{n\to\infty} \dfrac{\reg(I^{(n)})}{n} = \delta(I),$$
as required.
\end{proof}

\begin{rem} Although the limits $\lim_{n\to\infty} \dfrac{\omega(I^{(n)})}{n}$ and $\lim_{n\to\infty} \dfrac{\reg(I^{(n)})}{n}$ do exist, it is not true that the limit
$$
\lim_{n\to\infty} \dfrac{t_i(I^{(n)})}{n}
$$
exists for all $i\ge 0$.
\end{rem}

\begin{exm} In the polynomial ring $R = \Q[x,y,z,u,v]$, consider the ideal $I$ with the primary decomposition
$$I = (x^2,y^2,z^2)^2 \cap (x^3,y^3,u) \cap (z,v).$$
By \cite[Lemma 4.5]{NgT} we have 
$$
\depth(R/I^{(n)}) = 
\begin{cases}
1 &\text{ if } n \text{ is odd},\\
2 &\text{ if } n \text{ is even}.
\end{cases}
$$
From the Auslander-Buchsbaum formula, we get 
$$
\pd I^{(n)}=
\begin{cases}
3 &\text{ if } n \text{ is odd},\\
2 &\text{ if } n \text{ is even}.
\end{cases}
$$
In particular, $t_3(I^{(n)}) = -\infty$ if $n$ is even, and $t_3(I^{(n)}) >0$ if $n$ is odd. Since $t_3(I^{(n)})$ is a quasi-linear function in $n$ for $n\gg 0$, we deduce that
$$\liminf_{s\to\infty} \frac{t_3(I^{(2s+1)})}{2s+1} \geqslant 0.$$ 
So the limit $\lim_{n\to\infty} \dfrac{t_3(I^{(n)})}{n}$ does not exit.
\end{exm}
\begin{exm}
\label{exm_non-linear}
Let $p,q\ge 1$ be integers. Let $R=k[x,y,z,t]$, $I=I_{p,q}=(x,y) \cap (x,z^p) \cap (y^p,t^q)$. Explicitly
\[
I=(xy^p,xt^q,y^pz^p,yz^pt^q).
\]
On the one hand, it is not hard to compute $\delta(I)$. Indeed, $\smp(I)$ is defined by linear inequalities
\[
\begin{cases}
x+y \ge 1,\\
px+z \ge p,\\
qy+pt \ge pq,\\
x,y,z,t \ge 0.
\end{cases}
\]
Subtracting $z,t$ (and also $x$ and $y$ if necessary), to suitable non-negative numbers, we see that any vertex of $\smp(I)$ must satisfy the following system of equalities and inequalities
\[
\begin{cases}
x+y \ge 1,\\
px+z = p,\\
qy+pt = pq,\\
x,y,z,t \ge 0.
\end{cases}
\]
\end{exm}
This has solution $z=p(1-x),t=q(p-y)/p$ and
\[
\begin{cases}
x+y \ge 1,\\
0\le x \le 1,\\
0\le y\le p.
\end{cases}
\]
The last system yields a trapezoid in the $xy$ plane with vertices 
$$
(x,y) \in \{(1,0),(0,1),(0,p),(1,p)\}.
$$
Hence $\smp(I)$ has the following vertices
\[
(x,y,z,t) \in \{(1,0,0,q),(0,1,p,q(p-1)/p),(0,p,p,0),(1,p,0,0)\}.
\]
In particular, $\delta(I) \in \{q+1,p+q+1-q/p,2p\}$, and concretely
\[
\delta(I)=\begin{cases}
                 q+1, &\text{if $q \ge p^2+1$},\\
                 p+q+1-q/p, &\text{if $p^2\ge q \ge p$},\\
                 2p, &\text{if $q\le p-1$}.
                \end{cases}
\]
Therefore
\[
\lim_{n\to \infty} \frac{\reg I^{(n)}}{n}=\begin{cases}
                 q+1, &\text{if $q \ge p^2+1$},\\
                 p+q+1-q/p, &\text{if $p^2\ge q \ge p$},\\
                 2p, &\text{if $q\le p-1$}.
                \end{cases}
\]
On the other hand, the regularity function $I^{(n)}$ can be rather complicated in certain cases. For example, the above asymptotic formula suggests that when $q=p+1$, $\reg I_{p,p+1}^{(n)}$ seems to be eventually quasi-linear of periodic $p$. Experiments with Macaulay2 \cite{GS} also suggest that $I_{p,p+1}^{(n)}$ is not Koszul for all $p\ge 2$, $n\ge 1$, hence the techniques developed in the present paper do not apply to the computation of $\reg I_{p,p+1}^{(n)}$. We also see from the asymptotic formula that for $p\ge 2$, $\reg I_{p,p+1}^{(n)}$ is not eventually linear. Hence Question \ref{quest_2_linear_behavior} has a negative answer in embedding dimension 4 if we also take non-squarefree monomial ideals into account.

\section{Cover ideals}
\label{sect_coverideals}

In this section we investigate the symbolic powers of cover ideals of  graphs. Our main results in this section are:
\begin{enumerate}
 \item Theorem \ref{thm_delta_JG}, which determines explicitly the invariant $\delta(J(G))$ in terms of the combinatorial data of $G$;
 \item Theorem \ref{thm_delta_JG_equal_taumaxG}, which provides large families of graphs $G$ such that $\delta(J(G))$ attains its minimal value $\omega(J(G))$;
 \item Theorem \ref{thm_gen_degree_symbolicpowers}, which computes the maximal generating degrees of the symbolic powers of $J(G)$.
\end{enumerate}
Combining Theorems \ref{thm_delta_JG} and \ref{thm_gen_degree_symbolicpowers} with a result on the Koszul properties of the symbolic powers of some cover ideals, we construct in Theorem \ref{thm_non-linear} a family of graphs $G$ for which both $\reg J(G)^{(n)}$ and $\omega(J(G)^{(n)})$ are not eventually linear function of $n$.

Let $\Delta$ be a simplicial complex on the vertex set $\{1,\ldots,r\}$ and $n\geqslant 1$. We first describe $\smp_n(I_{\Delta})$ in a more specific way. For $F\in \mathcal F(\Delta)$, $NP(P_F^n)$ is defined by the system
$$\sum_{i\notin F} x_i \geqslant n, x_1\geqslant 0, \ldots, x_r\geqslant 0,$$ 
so that $\smp_n(I_{\Delta})$ is determined by the following system of inequalities:
\begin{equation}\label{EQ1}
\begin{cases}
\sum_{i\notin F} x_i \geqslant n, ~\text{for}~ F\in\mathcal F(\Delta),\\
x_1\geqslant 0,\ldots,x_r\geqslant 0.
\end{cases}
\end{equation}

From this, one has
\begin{rem}
\label{rem_GEN} Let $x^{\alb}\in I_{\Delta}^{(n)}$ be a monomial. The following are equivalent:

\begin{enumerate}
 \item $x^{\alb}\in \Gcc(I_{\Delta}^{(n)})$;
 \item for every $i$ such that $\langle \alb,\e_i\rangle\geqslant 1$, we have $x^{\alb-\e_i}\notin I_{\Delta}^{(n)}$,
 \item for every $i$ such that $\langle \alb,\e_i\rangle\geqslant 1$, there exists $F \in \Fcc(\Delta)$ such that $i\notin F$ and $ \left \langle \alb,\sum_{j\notin F}\e_j \right \rangle = n$.
\end{enumerate}
\end{rem}
The following lemma is a consequence of the last remark.
\begin{lem}
\label{lem_raising_power_mingens}
Let $p\geqslant 1$, $m_1,\ldots,m_p\geqslant 0$ be integers and $x^{\alb^j} \in I_{\Delta}^{(m_j)}$ be monomials for $j=1,\ldots,p$. Assume that $x^{\alb^1+\cdots+ \alb^p} \in I_{\Delta}^{(m_1)}\cdots I_{\Delta}^{(m_p)} \subseteq I_{\Delta}^{(m_1+\cdots +m_p)}$ is a minimal generator of $I_{\Delta}^{(m_1+\cdots+m_p)}$. Then for all $n_1,\ldots,n_p\geqslant 0$, $x^{n_1\alb^1+\cdots +n_p\alb^p}$ is a minimal generator of $I_{\Delta}^{(m_1n_1+\cdots+m_pn_p)}$.

In particular: 
\begin{enumerate}[\quad \rm(i)]
\item For every subset $W \subseteq [p]$, $x^{\sum_{i\in W} \alb^i}$ is a minimal generator of $I_{\Delta}^{(\sum_{i\in W}m_i)}$.
 \item If $x^{\alb}\in \Gcc(I_{\Delta})$ then $x^{n\alb} \in \Gcc(I_{\Delta}^{(n)})$ for all $n\geqslant 1$.
\end{enumerate}
\end{lem}
\begin{proof}
We claim that for every $1\leqslant i \leqslant p$, if $m_i=0$ then $\alb^i=\bf{0}$. Indeed, for example, assume $m_p=0$ and $\alb^p \neq \bf{0}$. Then
\[
x^{\alb^1+\cdots+ \alb^p}=x^{\alb^p}x^{\alb^1+\cdots+ \alb^{p-1}} \notin \Gcc(I_{\Delta}^{(m_1+\cdots+m_{p-1})})=\Gcc(I_{\Delta}^{(m_1+\cdots+m_p)}),
\]
a contradiction. Hence the claim is true. In view of the desired conclusion, we can assume that $m_i\geqslant 1$ for all $i=1,\ldots,p$. 

Take arbitrary $i$ such that $\langle n_1\alb^1+\cdots+n_p\alb^p,\e_i\rangle\geqslant 1$. Then $\langle \alb^1+\cdots+\alb^p,\e_i\rangle\geqslant 1$. Since $x^{\alb^1+\cdots+ \alb^p} \in \Gcc(I_{\Delta}^{(m_1+\cdots+m_p)})$, by Remark \ref{rem_GEN}, there exists $F\in \Fcc(\Delta)$ such that $i\notin F$ and 
\[
\left \langle \alb^1+\cdots+\alb^p,\sum_{j\notin F}\e_j \right \rangle = m_1+\cdots+m_p.
\]
For all $u=1,\ldots,p$, since $x^{\alb^u} \in I_{\Delta}^{(m_u)}$, 
\[
\left \langle \alb^u,\sum_{j\notin F}\e_j \right \rangle \geqslant m_u.
\]
Thus the equality actually happens for all $u=1,\ldots,p$. This implies that
\[
\left \langle n_1\alb^1+\cdots+n_p\alb^p,\sum_{j\notin F}\e_j \right \rangle = m_1n_1+\cdots+m_pn_p.
\]
Hence by Remark \ref{rem_GEN}, $x^{n_1\alb^1+\cdots +n_p\alb^p}$ is a minimal generator of $I_{\Delta}^{(m_1n_1+\cdots+m_pn_p)}$. The proof is concluded.
\end{proof}

\begin{lem}
\label{lem_upper_bound_omega_via_delta}
For all $n\geqslant 1$, there is an inequality $\omega(I_{\Delta}^{(n)}) \leqslant \delta(I_{\Delta}) n$.
\end{lem}
\begin{proof}
For simplicity, denote $\delta = \delta(I_{\Delta})$. Let $x^{\alb}$ be a minimal generator of $I_{\Delta}^{(n)}$. We may assume that $\alpha_i \geqslant 1$ for $i=1,\ldots, p$ and $\alpha_i = 0$ for $i = p+1,\ldots,r$ for some $1\leqslant p \leqslant r$.
 
 For each $i=1,\ldots, p$, there is a facet $F_i \in \Fcc(\Delta)$ which does not contain $i$ such that $\alb$ lies in the hyperplane $\sum_{j\notin F_i} x_j = n$. From the system \eqref{EQ1} we deduce that the intersection of $\smp_n(I_{\Delta})$ with the set
$$
\begin{cases}
\sum_{j\notin F_i} x_j = n \  & \text{ for } i=1,\ldots,p, \ \\
x_s = 0 \ & \text{ for } s = p+1,\ldots, r, 
\end{cases}
$$
is a compact face of $\smp_n(I_{\Delta})$. 

Since $\alb$ belongs to this face, there is a vertex $\gmb$ of $\smp_n(I_{\Delta})$ lying on this face such that $|\alb| \leqslant |\gmb|$. As $\gmb/n$ is a vertex of $\smp(I_{\Delta})$, $|\alb| \leqslant |\gmb| =  |\gmb/n|\cdot n \leqslant \delta n$. The conclusion follows.
\end{proof}

\begin{exm}
\label{ex_delta_edge_ideals}
Let $G$ be a graph on the vertex set $\{1,\ldots,r\}$. Let
$$
I(G) = (x_ix_j \mid \{i,j\} \in E(G)) \subseteq k[x_1,\ldots,x_r]
$$
be the edge ideal of $G$. Then $\omega(I(G)^{(n)}) =2n$ for all $n\geqslant 1$. In particular, by Theorem \ref{thm_limit_d}, $\delta(I(G))=2$.
 
Indeed, for any $n\geqslant 1$ we have $\omega(I(G)^{(n)}) \leqslant 2n$ by \cite[Corollary 2.11]{B}. On the other hand, if $x_ix_j$ is a minimal generator of $I(G)$, then $(x_ix_j)^n$ is a minimal generator of $I(G)^{(n)}$, and so $\omega(I(G)^{(n)}) \geqslant 2n$. Hence, $\omega(I(G)^{(n)}) = 2n$. 

Of course, $I(G)^{(n)}$ need not be generated in degree $2n$. For example, if $I(G)=(xy,xz,yz)$ then 
$$
I(G)^{(2)}=(x,y)^2\cap (x,z)^2\cap (y,z)^2=(x^2y^2,x^2z^2,y^2z^2,xyz).
$$

We do not know whether for any graph $G$, $\reg I(G)^{(n)}$ is asymptotically linear in $n$. This is the case when $G$ is a cycle (see \cite[Corollary 5.4]{GHOW}). 
\end{exm}

Let $G$ be a graph on $[r]=\{1,\ldots,r\}$. Then the polyhedron $\smp(J(G))$ is defined by the following system of inequalities:
\begin{equation*}
\begin{cases}
x_i+x_j \geqslant 1, \text{ for } \{i,j\} \in E(G),\\
x_1\geqslant 0,\ldots,x_r\geqslant 0.
\end{cases}
\end{equation*}

The following lemma is quite useful to identifying the vertices of $\smp(J(G))$.

\begin{lem}
\label{lem_vertex_smpJG}
Let $G$ be a graph on $[r]$ with no isolated vertex, and $\alb=(\alpha_1,\ldots,\alpha_r) \in \R^r$. Assume that $\alb$ is a vertex of $\smp(J(G))$. Then $\alpha_i\in \{0, 1/2,1\}$ for every $i=1,\ldots,r$. Denote $S_0=\{i: \alpha_i=0\}$, $S_1=\{i: \alpha_i=1\}$ and $S_{1/2}=\{i: \alpha_i=1/2\}$. Then the following statements hold:
\begin{enumerate}[\quad \rm(i)]
 \item $S_0$ is an independent set of $G$.
 \item $S_1=N(S_0)$.
 \item The induced subgraph of $G$ on $S_{1/2}$ has no bipartite component.
 \item If $v$ is a leaf not lying in $S_0$ and $N(v)=\{u\}$ then $u\notin S_1$.
\end{enumerate}
\end{lem}
\begin{proof}  
Since $\alb$ is a vertex of  $\smp(J(G))$, by \cite[Formula (23), Page 104]{S}, $\alb$ is the unique solution of a system
\begin{equation*}
\begin{cases}
x_i+x_j = 1, \text{ for } \{i,j\} \in E_1,\\
x_i = 0, \text{ for } i\in V_1,
\end{cases}
\end{equation*}
of exactly $r$ linearly independent equations, where $E_1\subseteq E(G)$ and $V_1 \subseteq \{1,\ldots, r\}$ with $|E_1|+|V_1| = r$. 

\textsf{Step 1:} Let $H$ be the subgraph of $G$ with the same vertex set and $E(H)=E_1$. Let $H_1,\ldots, H_s$ be connected components of $H$. Assume that $V(H_i) \cap V_1 \ne \emptyset$ for $i=1,\ldots,t$; and $V(H_i) \cap V_1=\emptyset$ for $i=t+1,\ldots,s$ for some $0\leqslant t \leqslant s$. We show that $\alpha_j \in \{0,1\}$ if $j\in \bigcup_{i=1}^t V(H_i)$ and $\alpha_j=1/2$ if $j\in \bigcup_{i=t+1}^s V(H_i)$.

For each $i\in \{1,\ldots,t\}$ and each $j\in H_i$, we take $p\in V(H_i) \cap V_1$. Then $\alpha_p = 0$ by the assumption. Since $H_i$ is connected, there is a path from $p$ to $j$ in $H_i$, say 
$$p = j_0, j_1,\ldots,j_m = j.$$
Since $\alpha_{j_u}+\alpha_{j_{u+1}}=1$ for $u=0,\ldots,m-1$, we deduce that $\alpha_{j_m} = \begin{cases}
                                                                                          0, & \text{if $m$ is even},\\
                                                                                          1, & \text{if $m$ is odd.}
                                                                                         \end{cases}$
For each $u=t+1,\ldots,s$, from the above discussion, the system
\begin{equation*}
\begin{cases}
x_i+x_j = 1,\\
\{i,j\} \in E(H_u),
\end{cases}
\end{equation*}
also has a unique solution. As $V(H_u)\cap V_1=\emptyset$, $H_u$ cannot be an isolated vertex, so $E(H_u) \neq \emptyset$. Since $x_i = 1/2$ for all $i\in V(H_u)$ is a solution of the last system, it is the unique one. Hence we see that $\alpha_i\in \{0,1,1/2\}$ for all $i$.

\textsf{Step 2:} If there are adjacent vertices $i,j\in S_0$ then as $\alb\in \smp(J(G))$, we get $0=\alpha_i+\alpha_j\ge 1$. This is a contradiction. Hence $S_0$ is an independent set, proving (i).

\textsf{Step 3:} Similarly there can be no edge connecting any $i\in S_0$ with some $j\in S_{1/2}$. Hence $N(S_0) \subseteq S_1$.

Now assume that $S_1$ has a vertex, say $i$, that is not adjacent to any vertex in $S_0$. Then $\gmb = \alb -\frac{1}{2}\e_i$ is a point of $\smp(J(G))$. On the other hand, $\alb +\frac{1}{2}\e_i$ is obviously a point of $\smp(J(G))$. Hence we have a convex decomposition
$$
\alb = \frac{1}{2}(\alb -\e_i/2) + \frac{1}{2}(\alb +\e_i/2),
$$
contradicting the fact that $\alb$ is a vertex of $\smp(J(G))$. Thus, as $G$ has no isolated vertex, every vertex in $S_1$ is adjacent to one in $S_0$, and thus $S_1 \subseteq N(S_0)$. In particular, $S_1=N(S_0)$, proving (ii).

\textsf{Step 4:} Next we show (iii). Assume the contrary, the induced subgraph of $G$ on $S_{1/2}$ has a bipartite component $G_1$. Let $(A,B)$ be the bipartition of $G_1$. Construct the vectors $\alb', \alb''$ as follows: $\alb'_i= \begin{cases} \alb_i &\text{if $i\notin A\cup B$},\\                                                                                                                                                                                                        0, &\text{if $i\in A$},\\                                                                                                                                                                                                            1, &\text{if $i\in B$}                                                                                                                                                                                                                            \end{cases}                                                                                                                                                                                                                            $
and  

$\alb''_i= \begin{cases} \alb_i &\text{if $i\notin A\cup B$},\\                                                                                                                                                                                                            1, &\text{if $i\in A$},\\                                                                                                                                                                                                            0, &\text{if $i\in B$}.                                                                                                                                                                                               \end{cases}                                                                                                                                                                                                  $
         
We show that $\alb',\alb'' \in \smp(J(G))$. Indeed, take an edge $\{i,j\}\in E(G)$. If neither $i$ nor $j$ belong to $A\cup B$, then $\alb'_i+\alb'_j=\alb_i+\alb_j\ge 1$. If exactly one of $i$ and $j$ belongs to $A\cup B$, we can assume that $i$ does. Then $j\in S_1$, since by (ii), $V(G_1) \subseteq S_{1/2} \subseteq V(G)\setminus N(S_0)$. In this case $\alb'_i+\alb'_j=\alb'_i+\alb_j=1+\alb'_i\ge 1$. If both $i$ and $j$ belong to $A\cup B$, then we can assume that $i\in A, j\in B$, so  $\alb'_i+\alb'_j=1$. Hence in any case $\alb'\in \smp(J(G))$, and the same argument works for $\alb''$.

But then the convex decomposition $\alb = (\alb'+\alb'')/2$  shows that $\alb$ is not a vertex of $\smp(J(G))$, a contradiction. Thus (iii) is true.

\textsf{Step 5:} Assume that $u\in S_1$. By part (ii), we get $S_0\neq \emptyset$. Since $v\notin S_0$, either $v\in S_1$ or $v\in S_{1/2}$. If $v\in S_1$ then by (ii), $v\in N(S_0)$, a contradiction with $v$ is a leaf and its unique neighbor is $u\in S_1$. Hence $v\in S_{1/2}$. Define the vectors $\alb^1$, $\alb^2$ as follows:

$$\alb^1_i= \begin{cases} \alb_i &\text{if $i\neq v$},\\                                                                                                                                                                                                        0, &\text{if $i=v$},                                                                                                                                                                                                                           \end{cases}                                                                                                                                                                                                                            $$
and 
$$\alb^2_i= \begin{cases} \alb_i &\text{if $i\neq v$},\\                                                                                                                                                                                                            1, &\text{if $i=v$}.                                                                                                                                                                                             \end{cases}                                                                                                                                                                                                     $$
Since $\alb^2 \ge \alb$ componentwise, $\alb^2\in \smp(J(G))$. We show that $\alb^1\in \smp(J(G))$. Take any edge $\{i,j\} \in E(G)$. If $i\neq v$ and $j\neq v$, then $\alb^1_i+\alb^1_j=\alb_i+\alb_j \ge 1$. If say $i=v$, then necessarily $j=u$, and
\[
\alb^1_i+\alb^1_j=\alb^1_v+\alb^1_u=0+\alb_u=1,
\]
noting that $u\in S_1$. Hence $\alb^1\in \smp(J(G))$. But then the convex decomposition $\alb = (\alb^1+\alb^2)/2$  shows that $\alb$ is not a vertex of $\smp(J(G))$, a contradiction. Thus (iv) is true and the proof is concluded.
\end{proof}

The first main result of this section is
\begin{thm} 
\label{thm_delta_JG} Let $G$ be a graph on $[r]$ without isolated vertices, and $J = J(G)$. Then there are equalities
\begin{align}
&\delta(J) = \label{eq_deltaJ}\\
&\max\left\{ |N(S)| + \frac{\left|G\setminus N[S]\right|}{2} \mid S \in \Delta(G) \text{ and } G\setminus N[S] \text{ has no bipartite component} \right\} \nonumber\\
&=\frac{r}{2} + \max\left\{ \frac{|N(S)| - |S|}{2} \mid S \in \Delta(G) \text{ and } G\setminus N[S] \text{ has no bipartite component}\right\}.\nonumber 
\end{align}

\end{thm}
\begin{proof} Let $d$ be the expression in the last line of \eqref{eq_deltaJ}. Clearly $d$ equals the expression on the second line of \eqref{eq_deltaJ}, as $r=|G|=|S|+|N(S)|+\left|G\setminus N[S]\right|$.

\textsf{Step 1:} We show that $d\le \delta(J)$. 

Let $S$ be an independent set of $G$ such that $d = r/2+(|N(S)|-|S|)/2$ and $G\setminus N[S]$ has no bipartite component.

For $i=1,\ldots,r$, define $\gamma_i$ as follows
$$\gamma_i = 
\begin{cases}
0  & \text{ if } i \in S,\\
1 & \text{ if } i \in N(S),\\
\frac{1}{2} & \text{ if } i \in V(G) \setminus N[S].
\end{cases}
$$
Let $\gmb = (\gamma_1,\ldots,\gamma_r)$. Then $\gmb$ is a point of $\smp(J)$. Since $2\gmb \in \N^r$, $x^{2\gmb} \in J(G)^{(2)}$. Observe that $x^{2\gmb}$ is a minimal generator of $J(G)^{(2)}$, since $G$ has no isolated vertex. Hence $|2\gmb|\leqslant 2\delta(J)$ by Lemma \ref{lem_upper_bound_omega_via_delta}, namely  $\delta(J) \geqslant |\gmb| = d$.

\textsf{Step 2:} To prove the reverse inequality, let $\alb = (\alpha_1,\ldots,\alpha_r)$ be any vertex of $\smp(J)$.  By Lemma \ref{lem_vertex_smpJG}, $\alpha_i \in \{0,1/2,1\}$ for every $i$. Let $S=S_0 = \{i \mid \alpha_i = 0\}$, $S_1 = \{i\mid \alpha_i = 1\}$ and $S_{1/2} = \{i\mid \alpha_i = 1/2\}$. By the same lemma, $S\in \Delta(G)$ and $G\setminus N[S]$ has no bipartite component.

Thus 
$$
|\alb| = |S_1| +\frac{|S_{1/2}|}{2}= \frac{|S|+|S_1|+|S_{1/2}|}{2}+\frac{|S_1|-|S|}{2}=\frac{r}{2}+\frac{|N(S)|-|S|}{2}\leqslant d.
$$
Choosing the vertex  $\alb$ such that $|\alb|=\delta(J)$, we deduce $\delta(J) \leqslant d$, as required.
\end{proof}
Denote by $\tau_{\max}(G)$ the maximal cardinality of a minimal vertex cover of $G$. Since the minimal monomial generators of $J(G)$ correspond to the minimal vertex covers of $G$, there is an equality $\omega(J(G))=\tau_{\max}(G)$.
\begin{cor}
\label{cor_bounds_for_deltaJG}
Let $G$ be a graph on $[r]$ without isolated vertices. Then there are inequalities 
$$
\max\left\{\tau_{\max}(G),\frac{r}{2}\right \} \le \delta(J(G)) \le \max\left\{\tau_{\max}(G),\frac{r}{2},\frac{r+\tau_{\max}(G)-3}{2}\right \}.
$$
\end{cor}
\begin{proof}
By Lemma \ref{lem_upper_bound_omega_via_delta} for $n=1$, $\tau_{\max}(G)=\omega(J(G))\le \delta(J(G))$.

We note that $x_1x_2\cdots x_r \in J(G)^{(2)}$. It is a minimal generator of $J(G)^{(2)}$, since $G$ has no isolated vertex. Hence again by Lemma \ref{lem_upper_bound_omega_via_delta},
\[
r \le \omega(J(G)^{(2)}) \le 2\delta(J(G)),
\]
namely $r/2 \le \delta(J(G))$. This yields the inequality on the left.

For the inequality on the right, take any independent set $S$ of $G$ such that $G\setminus N[S]$ has no bipartite component. We have to show that
\begin{equation}
\label{eq_ineq_bounding_above_deltaJG}
\frac{r}{2}+\frac{|N(S)|-|S|}{2} \le  \max\left\{\tau_{\max}(G),\frac{r}{2},\frac{r+\tau_{\max}(G)-3}{2}\right\}.
\end{equation}
If $S=\emptyset$ then the left-hand side is $r/2$. If $G\setminus N[S]=\emptyset$, then $r=|N(S)|+|S|$. In this case the left-hand side of \eqref{eq_ineq_bounding_above_deltaJG} is $|N(S)| \le \tau_{\max}(G)$, since $N(S)$ is now a minimal vertex cover of $G$.

Assume that $S$ and $G\setminus N[S]$ are both non-empty. Let $H$ be a connected component of $G\setminus N[S]$, then by the assumption on $S$, $H$ is neither an isolated point, nor bipartite. Thus $|V(H)|\ge 3$. As a connected graph, $H$ has then a minimal vertex cover $W$ of size at least 2. For this, note that if $H$ has a minimal vertex cover of a singleton $u$, then $H=N_H[u]$, and  $N_H(u)$ is an independent set. In turn, this implies that $N_H(u)$ is a minimal vertex cover of $H$ of size $|V(H)|-1\ge 2$.

Let $W'$ be a minimal vertex cover of $G\setminus N[S]$ containing $W$. Then $N(S)\cup W'$ is a minimal vertex cover of $G$ (the minimality holds since $S$ is an independent set). Thus
\[
|N(S)| \le  \tau_{\max}(G)-|W'| \le \tau_{\max}(G)-2.
\]
Consequently, using the fact that $S\neq \emptyset$,
\[
\frac{r}{2}+\frac{|N(S)|-|S|}{2} \le \frac{r}{2}+\frac{\tau_{\max}(G)-2-1}{2}=\frac{r+\tau_{\max}(G)-3}{2}.
\]
This finishes the proof of \eqref{eq_ineq_bounding_above_deltaJG}, and that of the corollary.
\end{proof}
\begin{rem}
Computations with Macaulay2 \cite{GS} show that for any graph $G$ without isolated vertex on $r\le 8$ vertices, the equality $\delta(J(G))=\tau_{\max}(G)$ holds. In particular, for such graphs, $\delta(J(G))=\max\{\tau_{\max}(G),|V(G)|/2\}$.

The corona graph $G=\cg(K_3,2)$ in Figure \ref{fig_corona} has 9 vertices, and $\delta(J(G))=9/2> \tau_{\max}(G)=4$, hence again $\delta(J(G))=\max\{\tau_{\max}(G),|V(G)|/2\}$.

In general, both inequalities in Corollary \ref{cor_bounds_for_deltaJG} are strict, see Example \ref{ex_strict_inequalities}.
\end{rem}

We will see later in Lemma \ref{lem_deg_delta_corKms} a family of graphs for which the difference $\delta(J(G))-\tau_{\max}(G) $ can be arbitrarily large. On the other hand, for large classes of graphs, the equality $\delta(J(G))=\tau_{\max}(G)$ does hold. We say that $G$ is an \emph{unmixed} graph if every minimal vertex cover of $G$ has the same size. Equivalently, $G$ is unmixed if and only if every associated prime ideal of $J(G)$ has the same height. We say that $G$ is \emph{claw-free} if it does not contain the complete bipartite graph $K_{1,3}$ as an induced subgraph.

\begin{figure}[ht]
\centering
 \includegraphics[scale=0.7]{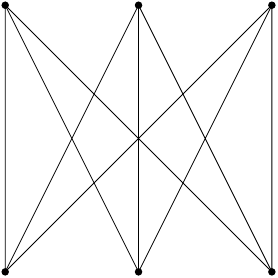}
\caption{The graph $K_{3,3}$ is bipartite, unmixed, but not claw-free}
\label{fig_K33}
\end{figure}

The second main result of this section is
\begin{thm}
\label{thm_delta_JG_equal_taumaxG} 
Let $G$ be a graph without isolated vertices, that satisfies either of the following properties:
\begin{enumerate}[\quad \rm(1)]
 \item bipartite;
 \item unmixed;
 \item claw-free.
\end{enumerate}
Then there are equalities $\delta(J(G))=\omega(J(G))=\tau_{\max}(G)$.
\end{thm}
The proof uses the following lemma, that is inspired by work of Seyed Fakhari \cite[Theorem 3.2]{SF2}.
\begin{lem}
\label{lem_sufficient_delta_JG_equal_taumaxG}
Let $\Hcc$ be a family of graphs with the following properties:
\begin{enumerate}[\quad \rm (i)]
 \item for every $G\in \Hcc$ and every vertex $x\in G$, the graph $G\setminus N_G[x]$ also belongs to $\Hcc$;
 \item for every $G\in \Hcc$ without isolated vertices, the inequality $\tau_{\max}(G)\ge \frac{|V(G)|}{2}$ holds.
\end{enumerate}
Then for every $G\in \Hcc$ without isolated vertices, the equality $\delta(J(G))=\tau_{\max}(G)$ holds.
\end{lem}
\begin{proof}
By Corollary \ref{cor_bounds_for_deltaJG}, it remains to show that for any $G\in \Hcc$ without isolated vertices, $\delta(J(G)) \le \tau_{\max}(G)$. We use the formula of Theorem \ref{thm_delta_JG}.

Let $S$ be an independent set of $G$ such that $S'=G\setminus N[S]$ has no bipartite component. Applying successively property (i) of $\Hcc$, $S'$ also belongs to $\Hcc$. Since $S'$ has no bipartite component, it has no isolated vertices. In particular, property (ii) guarantees the existence of a minimal vertex cover $W$ of $S'$ with cardinality $\ge |S'|/2$.

It is a general fact that when $S$ is an independent set, and $W$ is a minimal vertex cover of $G\setminus N[S]$, $W\cup N(S)$ is a minimal vertex cover of $G$. In our case, $W\cup N(S)$ has cardinality $\ge |N(S)|+|S'|/2$. This shows that
\[
\tau_{\max}(G) \ge |N(S)|+|S'|/2=|N(S)|+\frac{|G\setminus N[S]|}{2}.
\]
Taking supremum over all $S$, we deduce $\tau_{\max}(G) \ge \delta(J(G))$. The proof of the lemma is concluded.
\end{proof}
\begin{proof}[Proof of Theorem \ref{thm_delta_JG_equal_taumaxG}]
We check that each of the families of graphs in Theorem \ref{thm_delta_JG_equal_taumaxG} satisfies the two conditions in Lemma \ref{lem_sufficient_delta_JG_equal_taumaxG}. This was done respectively in the proofs of Theorems 3.4, 3.6 and 3.7 in \cite{SF2}.
\end{proof}
The following result gives a family of graphs $G$ with the property $\delta(J(G))=\tau_{\max}(G)$ not covered by Theorem \ref{thm_delta_JG_equal_taumaxG}. Recall that for two graphs $G$ and $H$, their \emph{join} $G*H$ is the graph with vertices $V(G) \bigsqcup V(H)$ and edges $E(G) \bigsqcup E(H) \bigsqcup \{\{x,y\}: x\in V(G), y\in V(H)\}$. Proposition \ref{prop_non_bipartite_delta_JG_equal_taumaxG} says that for ``most'' sparse enough graphs $G$, $G$ is an induced subgraph of a graph $G_m$ having $\delta(J(G_m))=\tau_{\max}(G_m)$, but having none of the properties bipartite, unmixed, and claw-free.
\begin{prop}
\label{prop_non_bipartite_delta_JG_equal_taumaxG}
Let $G$ be a graph without isolated vertices on $[r]$, where $r\ge 8$. Assume that $\tau_{\max}(G)\le r-4$, equivalently, every \textup{maximal} independent set of $G$ has size at least $4$. Let $m$ be any integer such that
\[
3\le m \le  \min \left\{\frac{r-\tau_{\max}(G)+3}{2}, \frac{r}{2}\right\}.
\]
Let $O(m)$ be the graph on $\{y_1,\ldots,y_m\}$ with no edges. Denote by $G_m$ the join of $G$ and $O(m)$. Then: 
\begin{enumerate}[\quad \rm (1)]
 \item  $G_m$ is a connected graph, but has none of the properties bipartite, unmixed, and claw-free;
 \item $\delta(J(G_m))=\tau_{\max}(G_m)=r$.
\end{enumerate}
\end{prop}
\begin{rem}
The hypothesis $\tau_{\max}(G)\le r-4$ of Proposition \ref{prop_non_bipartite_delta_JG_equal_taumaxG} is satisfied, for example, if $r=8$, and $G=K_{4,4}$, which has $\tau_{\max}(K_{4,4})=4$.  It is also satisfied if $r\ge 10$ and every vertex of $G$ has degree strictly less than $\frac{r}{3}-1$ (more concretely, the cycle of length $r$ has this property). Indeed, in that case, we show that $G$ has no maximal independent set $S$ of size $3$. Assume the contrary, that $S=\{x,y,z\}$ is a maximal independent set of size $3$. Then as $G$ has no isolated vertex, each of the $r-3$ vertices of $G \setminus S$ is adjacent to an element of $S$. This implies that $S$ has a vertex of degree at least $\frac{r}{3}-1$, a contradiction.
\end{rem}
\begin{proof}[Proof of Proposition \ref{prop_non_bipartite_delta_JG_equal_taumaxG}]
Denote $V(G)=\{x_1,\ldots,x_r\}$. 

(1) Clearly $G_m$ is connected as every two non-adjacent vertices of $G_m$ can be connected via either $y_1$ or $x_1$. Since $r\ge 2$, $G$ has at least one edge $x_ix_j$, $i\neq j$. Hence having the odd cycle $x_i,x_j,y_1$, $G$ is not bipartite.

Observation: $G_m$ has only two types of minimal vertex covers: $\{x_1,\ldots,x_r\}$, and $W\cup \{y_1,\ldots,y_m\}$, where $W$ is a minimal vertex cover of $G$.

The vertex cover $\{x_1,\ldots,x_r\}$ has cardinality $r >m+\tau_{\max}(G)$. The last inequality holds since from the hypotheses 
$$
m+\tau_{\max}(G) \le  \frac{r+\tau_{\max}(G)+3}{2} < \frac{r+(r-3)+3}{2}=r. 
$$ 
Hence $G_m$ is not unmixed and
\begin{equation}
\label{eq_tauGm}
\tau_{\max}(G_m)= r.
\end{equation}

As $m\ge 3$, $G_m$ contains the claw $x_1,y_1,y_2,y_3$, as desired.

(2) The equality $\tau_{\max}(G_m)= r$ is \eqref{eq_tauGm}.  By Corollary \ref{cor_bounds_for_deltaJG}, it remains to show that $\delta(J(G_m)) \le r$. Let $S$ be an independent set of $G_m$ such that $G_m \setminus N_{G_m}[S]$ has no bipartite component. Being independent, $S$ cannot intersect non-trivially with both $\{x_1,\ldots,x_r\}$ and $\{y_1,\ldots,y_m\}$. There are three cases.

\textbf{Case 1:} $S=\emptyset$. Using the formula of Theorem \ref{thm_delta_JG}, we get the value
\[
\frac{m+r}{2}+\frac{|N_{G_m}(S)|-|S|}{2}=\frac{m+r}{2} \le r.
\]
The inequality holds because $m\le r/2 <r$.

\textbf{Case 2:} $\emptyset \neq S\subseteq \{y_1,\ldots,y_m\}$. In this case, $N_{G_m}(S)=\{x_1,\ldots,x_r\}$. So $G_m \setminus N_{G_m}[S]$ is a subset of $\{y_1,\ldots,y_m\}$, and it has no bipartite component. In particular, it must be empty, so that $S=\{y_1,\ldots,y_m\}$. The formula of Theorem \ref{thm_delta_JG} yields the value
\[
\frac{m+r}{2}+\frac{|N_{G_m}(S)|-|S|}{2} = \frac{m+r}{2}+\frac{r-m}{2}=r.
\]

\textbf{Case 3:} $\emptyset \neq S\subseteq \{x_1,\ldots,x_r\}$. In this case $S$ is an independent set of $G$, $N_{G_m}(S)=N_G(S)\cup \{y_1,\ldots,y_m\}$. Hence $G \setminus N_G[S]=G_m \setminus N_{G_m}[S]$ has no bipartite component. The formula of Theorem \ref{thm_delta_JG} yields the value
\begin{align*}
\frac{m+r}{2}+\frac{|N_{G_m}(S)|-|S|}{2}&=\frac{m+r}{2}+\frac{m+|N_{G}(S)|-|S|}{2}\\
                                        &=m+\frac{r}{2}+\frac{|N_{G}(S)|-|S|}{2} \\
                                         &\le m+\delta(J(G)).
\end{align*}
Using Corollary \ref{cor_bounds_for_deltaJG}, we further get
\begin{align*}
 m+\delta(J(G)) &\le m+\max\left\{\tau_{\max}(G),\frac{r}{2},\frac{r+\tau_{\max}(G)-3}{2}\right \} \\
                &\le r,
\end{align*}
thanks to the hypothesis
\[
m \le  \min \left\{\frac{r-\tau_{\max}(G)+3}{2}, \frac{r}{2}\right\}.
\]
Hence in any case,  $\delta(J(G_m)) \le r$. The proof is concluded.
\end{proof}

The third main result in this section is
\begin{thm}
\label{thm_gen_degree_symbolicpowers} 
Let $G$ be a graph on $[r]$ with no isolated vertex, and $J = J(G)$. Then 
\begin{enumerate}[\quad \rm(1)]
\item $\omega(J^{(2s)}) = \delta(J)2s$ for every $s\geqslant 1$.
\item There exists $m_2 \in \Gcc(J)$ such that for some $m_1\in \Gcc(J^{(2)})$ satisfying $\deg(m_1)=2\delta(J)$, we have $m_1m_2 \in \Gcc(J^{(3)})$. Let $e$ be the maximal degree of such an $m_2$. Then  $$\omega(J^{(2s+1)}) = \delta(J)2s + e, \  \text{ for every } s\geqslant \omega(J)-e.$$ 
\item If $\delta(J) = \omega(J)$ or $\delta(J) = r/2$, then with the notation of part \textup{(2)}, we have $e=\omega(J)$. In particular, $\omega(J^{(2s+1)}) = \delta(J) 2s + \omega(J)$  for all $s \geqslant 0$.
\end{enumerate}
\end{thm}
It is crucial for the proof that the symbolic Rees algebras of cover ideals of graphs are generated in degree at most 2.

\begin{thm}[Herzog-Hibi-Trung {\cite[Theorem 5.1]{HHT}}]
\label{thm_symbolpower_JG} Let $G$ be a graph and $J = J(G)$. Then for every $s\geqslant 1$,
\begin{enumerate}[\quad \rm(1)]
\item $J^{(2s)} = (J^{(2)})^s$.
\item $J^{(2s+1)} = J(J^{(2)})^s$.
\end{enumerate}
\end{thm} 

For ease of reference, we record here an immediate corollary of this theorem.

\begin{cor}  
Let $G$ be a graph. Then $\reg J(G)^{(n)}$ is a quasi-linear function of $n$ of period at most $2$ for $n$ large enough.
\end{cor}
\begin{proof} Follows from Theorem \ref{thm_symbolpower_JG} and \cite[Theorem 3.2]{TW}.
\end{proof}

\begin{proof}[Proof of Theorem \ref{thm_gen_degree_symbolicpowers}]
(1) By Lemma \ref{lem_upper_bound_omega_via_delta},  $\omega(J^{(2s)})\le \delta(J)2s$.

For the reverse inequality, let $\alb=(\alpha_1,\ldots,\alpha_r)$ be a vertex of $\smp(J)$ such that $\delta(J) = |\alb|$. By \cite[Formula 23 in Page 104]{S}, $\alb$ is a unique solution of a system of the type
$$
\begin{cases}
x_i+x_j = 1, \text{ for } \{i,j\} \in E_1,\\
x_i = 0, \text{ for } i\in V_1,
\end{cases}
$$
where $E_1\subseteq E(G)$ and $V_1 \subseteq \{1,\ldots, r\}$ with $|E_1|+|V_1| = r$.  By Lemma \ref{lem_vertex_smpJG}, $\alpha_i \in \{0,1/2,1\}$ for every $i$.

Since $2s\alb \in \N^r$, we get $x^{2s\alb} \in J(G)^{(2s)}$. Note that $2s\alb$ is a vertex of $\smp_{2s}(J)$, so $x^{2s\alb}$ is a generator of $J(G)^{(2s)}$. It follows that $\omega(J^{(2s)}) \geqslant 2s|\alb| = \delta(J)2s$, as desired.

(2) Let $I = J^{(2)}$.  By Theorem \ref{thm_symbolpower_JG} we have $J^{(2s+1)} = I^s J$. Note that $\omega(I) = 2\delta(J)$ by part (1) above. Therefore, we can write $I = I_1+I_2$ where $I_2$ is generated by elements of $\Gcc(I)$ of degree exactly $2\delta(J)$ and $I_1$ is  generated by the remaining elements.

Since $J^{(3)}=IJ$, the first assertion of (2) reduces to the following

\textbf{Claim}: $I_2J \not\subseteq I_1J$. 

Indeed, if $I_2J \subseteq I_1J$, we will derive a contradiction. Since $IJ = (I_1+I_2)J = I_1J+I_2J = I_1J$, for every $n\geqslant 1$, from this equality and Theorem \ref{thm_symbolpower_JG} we get $J^{(2n+1)} = I^n J = I_1^nJ$. In particular, $\omega(J^{(2n+1)}) \leqslant \omega(I_1)n + \omega(J)$, so $\omega(I_1) \geqslant 2\delta(J)$ by Theorem \ref{thm_limit_d}. On the other hand, $\omega(I_1)  < 2\delta(J)$ by the definition of $I_1$, a contradiction.

We now return to proving part (2). By the claim, there exist $m_1 \in \Gcc(I)$ with $\deg(m_1)=2\delta(J)$ and $m_2 \in \Gcc(J)$ such that 
$m_1m_2\in \Gcc(IJ)$. Among all such couples $(m_1,m_2)$, choose one such that $e=\deg(m_2)$ is maximal. By Lemma \ref{lem_raising_power_mingens}, $m_1^s m_2 \in \Gcc(J^{(2s+1)})$. In particular, 
\begin{equation}
\label{eq_lowerbound_dJ2s+1}
\omega(J^{(2s+1)}) \geqslant \deg(m_1)s+\deg(m_2)=\delta(J)2s+e.
\end{equation}
It remains to show that the equality occurs whenever  $s \geqslant \omega(J)-e$.

Fix an $s \geqslant \omega(J)-e$. As mentioned above $J^{(2s+1)} = I^sJ$, so any minimal generator of $\Gcc(J^{(2s+1)})$ must have the form $g_1g_2\cdots g_s f$, where $g_i$ is a minimal generator of $I=J^{(2)}$ and $f$ is a minimal generator of $J$. Let $g=g_1\cdots g_s$. We can choose $g_i,f$ such that $\deg(gf)=\omega(J^{(2s+1)})$.

From Lemma \ref{lem_raising_power_mingens}, $g_if$ is a minimal generator of  $J^{(3)}=IJ$ and $g_i^sf\in \Gcc(J^{(2s+1)})$ for all $i$. Assume that $\deg g_1\leqslant \cdots\leqslant \deg g_s$. Then
\begin{align*}
\omega(J^{(2s+1)}) &= \deg(gf)=\deg g_1+\cdots+\deg g_s + \deg f\\
              &\leqslant s \deg g_s+\deg f \leqslant \omega(J^{(2s+1)}),
\end{align*}
so that $\deg g_1=\cdots=\deg  g_s$ and
\begin{equation}
\label{eq_deg_gf}
\omega(J^{(2s+1)})=\deg(gf) = s \deg(g_i)+\deg(f) \ \text{ for } i =1,\ldots,s.
\end{equation}
 If $\deg(g_1) = 2\delta(J)$, then thanks to \eqref{eq_deg_gf}, we get the inequality $\deg(f)\ge e$. The latter is necessarily an equality by the definition of $e$. In this case, $\omega(J^{(2s+1)}) = \delta(J)2s+e$. 

Assume that $\deg(g_1) < 2\delta(J)$. Since $s \geqslant \omega(J)-e \geqslant \deg(f) - \deg(m_2)$, we have
\begin{align*}
\delta(J)2s + \deg(m_2) &\geqslant (\deg(g_1)+1)s+ \deg(f)-s\\
                        &= \deg(g_1)s +\deg(f) = \omega(J^{(2s+1)}),
\end{align*}
so thanks to \eqref{eq_lowerbound_dJ2s+1}, $\omega(J^{(2s+1)})=  \delta(J)2s+e$, as required.

(3) If $\delta(J) = \omega(J)$, then there exists $m\in \Gcc(J)$ of degree $\delta(J)$. By Lemma \ref{lem_raising_power_mingens}, $m^3$ is a minimal generator of $J^{(3)}$, so we can choose $m_1=m^2, m_2=m$ and $e=\deg(m)=\delta(J)=\omega(J)$. 

If $\delta(J) = r/2$, then for $\alb = (1,\ldots,1)\in \N^r$, $x^{\alb} \in \Gcc(I)$ and $|\alb|=r=2\delta(J)$. Let $x^{\gmb}\in \Gcc(J)$ be such that $|\gmb| = \omega(J)$. Observe that $x^{\alb+\gmb}\in \Gcc(J^{(3)})$. Clearly $x^{\alb+\gmb} \in IJ\subseteq J^{(3)}$. For any $1\le i\le r$, we need to show that $x^{\alb+\gmb-\e_i} \notin J^{(3)}$. Note that $\gmb_j\in \{0,1\}$ for all $j=1,\ldots,r$. 

If $\gmb_i=1$, since $x^{\gmb}\in \Gcc(J)$, there is an edge $ij\in E(G)$ such that $\gmb_i+\gmb_j=1$. But then $\alb_i+\alb_j+\gmb_i+\gmb_j=3$, hence $x^{\alb+\gmb-\e_i} \notin J^{(3)}$. 

If $\gmb_i=0$, for any edge $ij\in E(G)$, we get $\alb_i+\alb_j+\gmb_i+\gmb_j=3$. Again $x^{\alb+\gmb-\e_i} \notin J^{(3)}$. Therefore we always have $x^{\alb+\gmb}\in \Gcc(J^{(3)})$. Hence we can choose $m_1=x^{\alb}, m_2=x^{\gmb}$ and again $e=\deg(m_2)=\omega(J)$.
\end{proof}

The following example shows that $\omega(J(G)^{(2n+1)})$ need not be a linear function in $n$ from $n =  0$, and the number $e$ in Theorem \ref{thm_gen_degree_symbolicpowers}(2) can be strictly smaller than $\omega(J)$.

\begin{exm} 
\label{ex_strict_inequalities}
Let $G$ be a graph with the vertex set 
$$
\{x_i,y_i,z_i \mid i=1,\ldots,5\}~  \cup ~ \{u,v,w\}
$$ 
depicted in Figure \ref{fig_18vertices}. Using the \textsf{EdgeIdeals} package in Macaulay2 \cite{GS}, the graph $G$ and its cover ideal are given as follows.
\begin{verbatim}
R=ZZ/32003[x_1..x_5,y_1..y_5,z_1..z_5,u,v,w];
G=graph(R,{x_1*x_2,x_1*x_3,x_1*x_4,x_1*x_5,x_2*x_3,x_2*x_4,
x_2*x_5,x_3*x_4,x_3*x_5,x_4*x_5,x_1*y_1,x_1*z_1,x_2*y_2,x_2*z_2,
x_3*y_3,x_3*z_3,x_4*y_4,x_4*z_4,x_5*y_5,x_5*z_5,x_3*u,x_4*u,y_5*u,
u*v,u*w,v*w});
J=dual edgeIdeal G
\end{verbatim}
In particular, $G$ has 18 vertices and 26 edges.

\begin{figure}[ht]
\centering
\includegraphics[scale=0.7]{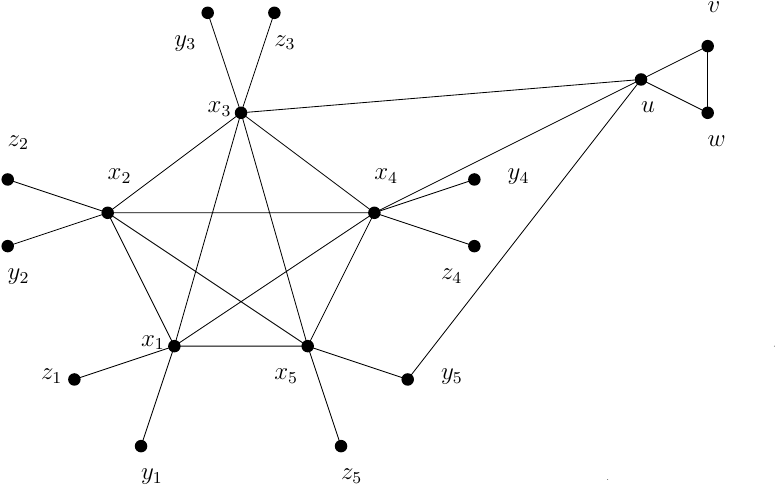}
\caption{The graph $G$}
\label{fig_18vertices}
\end{figure}
\end{exm}

Let $J = J(G)$. By using Macaulay2 \cite{GS} we get 
\begin{enumerate}
\item $\omega(J) = 9$, $\omega(J^{(2)}) = 19$, and $\omega(J^{(3)}) = 27$. By Theorem \ref{thm_gen_degree_symbolicpowers}, $\delta(J) = 19/2$.
\item The monomials $m_1 = u^2v^2\prod_{i=1}^5(x_i y_iz_i) \in \Gcc(J^{(2)})$ and 
$$
m_2=x_2x_3x_4x_5y_1z_1uv \in \Gcc(J)
$$ 
satisfy $m_1m_2 \in \Gcc(J^{(3)})$. Note that $\deg(m_1)=19, \deg(m_2)=8$.
\end{enumerate}
In the notation of Theorem \ref{thm_gen_degree_symbolicpowers}, we deduce $8\le e \le \omega(J)=9$. If $e=9$, then by \emph{ibid.} we have $\omega(J^{(2n+1)}) = \delta(J)2n+9 = 19n+9$ for $n\geqslant \omega(J) - 9 = 0$. Setting $n=1$, we get $\omega(J^{(3)})=28$, a contradiction. Hence $e=8$ and $\omega(J^{(2n+1)}) = 19n+8$ if (and only if) $n\geqslant \omega(J) - 8 = 1$.

Moreover, observe that both inequalities of Corollary \ref{cor_bounds_for_deltaJG} are strict in this case:
\begin{align*}
&\max\left\{\tau_{\max}(G),\frac{|V(G)|}{2}\right \} =\tau_{\max}(G)=9,\\
&\delta(J(G))=19/2,\\
& \max\left\{\tau_{\max}(G),\frac{|V(G)|}{2},\frac{|V(G)|+\tau_{\max}(G)-3}{2}\right \}=12.
\end{align*}

\section{The Koszul property of symbolic powers of  cover ideals}
\label{sect_Koszulproperty}

The following result is our main tool in the study of the Koszul property of symbolic powers.
\begin{thm}
\label{thm_ysplit}
Let $(R,\mm)$ be a standard graded $k$-algebra. Let $x$ be a non-zero linear form and $I', T$ be non-trivial homogeneous ideals of $R$ such that the following conditions are fulfilled:
\begin{enumerate}[\quad \rm(i)]
 \item $I'$ is a Koszul module and $x$ is $I'$-regular \textup{(}e.g. $x$ is an $R$-regular element\textup{)},
 \item $T\subseteq \mm I'$,
 \item $x$ is a regular element with respect to $R/T$ and $\gr_{\mm}T$.
 \end{enumerate}
Denote $I=xI'+T$. Then the decomposition $I=xI'+T$ is a Betti splitting, and there is a chain
$$
\lind_R T \le \lind_R I = \lind_R(T+(x))=\lind_R (T/xT) \le \lind_R T+1.
$$
Moreover, $I$ is a Koszul module if and only if $T$ is so.
\end{thm}

Before proving Theorem \ref{thm_ysplit}, we recall the following result.
\begin{lem}[Nguyen {\cite[Theorem 3.1]{Ng}}]
\label{lem_relativeext}
Let $0\to M'\to P' \to N' \to 0$ be a short exact sequence of non-zero finitely generated $R$-modules where
\begin{enumerate}[\quad\rm(i)]
\item $M'$ is a Koszul module;
\item $M'\cap \mm P'=\mm M'$.
\end{enumerate}
Then there are inequalities $\lind_R P'\le \lind_R N' \le \max\{\lind_R P',1\}$. In particular, $\lind_R N'=\lind_R P'$ if $\lind_R P'\ge 1$ and $\lind_R N'\le 1$ if $\lind_R P'=0$.

Moreover, $\lind_R N'=0$ if and only if $\lind_R P'=0$ and for all $s\ge 1$, we have $M'\cap \mm^sP'=\mm^sM$.
\end{lem}
We also have an easy observation.
\begin{lem}
\label{lem_regular_lin_form}
 Let $(R,\mm)$ be a standard graded $k$-algebra, and $x\in \mm$ a non-zero linear form. Let $T$ be a homogeneous ideal of $R$ such that $x$ is $(R/T)$-regular. Then the following are equivalent:
 \begin{enumerate}[\quad \rm(1)]
  \item $x$ is $\gr_\mm T$-regular,
  \item $\mm^sT:x=\mm^{s-1}T$ for all $s\ge 1$.
 \end{enumerate}
\end{lem}
\begin{proof}
Clearly $x$ is $\gr_\mm T$-regular if and only if
\begin{equation}
\label{eq_grmT-regular}
(\mm^{s+2}T:x)\cap \mm^sT= \mm^{s+1}T, \quad \text{for all $s\ge 0$}.
\end{equation}
Hence (2) $\Longrightarrow$ (1).

Conversely, assume that (1) is true. Since $\mm^{s-1}T \subseteq \mm^sT:x$, it suffices to show for all $s\ge 1$ that $\mm^sT:x \subseteq \mm^{s-1}T$. Induct on $s\ge 1$.

For $s=1$, 
\[
\mm T:x \subseteq T:x=T,
\]
where the equality follows from the hypothesis $x$ is $(R/T)$-regular.

Assume that the statement holds true for $s\ge 1$. Using the induction hypothesis, we have
\[
\mm^{s+1}T:x \subseteq (\mm^{s+1}T:x)\cap (\mm^s T:x) \subseteq (\mm^{s+1}T:x) \cap \mm^{s-1}T=\mm^s T.
\]
The equality in the chain follows from \eqref{eq_grmT-regular}. The proof is concluded.
\end{proof}

\begin{proof}[Proof of Theorem \ref{thm_ysplit}]
We proceed through several steps.

\textsf{Step 1:} First we establish the equalities $\lind_R I=\lind_R T/xT=\lind_R (T+(x))$. Consider the short exact sequence
\[
0 \longrightarrow xI' \longrightarrow I \longrightarrow \frac{T}{xI'\cap T}=\frac{T}{xT}\longrightarrow 0.
\]
The equality holds since $T:x=T\subseteq I'$.

We claim that
\begin{equation}
\label{eq_intersect}
xI' \cap \mm^s I=\mm^s xI'  \quad \text{for all $s\ge 1$}.
\end{equation}
The inclusion ``$\supseteq$'' is clear. For the converse inclusion, take $a\in xI'\cap \mm^s I=xI' \cap \mm^s(xI'+T)$. Subtracting to an element in $\mm^s xI'$, we may assume that $a\in xI' \cap \mm^s T$. The last module is contained in 
$$
(x) \cap \mm^sT=x(\mm^sT:x) \subseteq x\mm^{s-1}T \subseteq x\mm^s I'.
$$
In the last chain, the first inclusion follows from Lemma \ref{lem_regular_lin_form} and the hypothesis $x$ is $\gr_\mm T$-regular. The second inclusion follows from the hypothesis $T\subseteq \mm I'$. Thus the claim follows.

Recall that $xI'\cong I'$ is Koszul by the hypothesis. Hence using Lemma \ref{lem_relativeext} for the above exact sequence, and Equality \eqref{eq_intersect}, we get
\[
\lind_R I = \lind_R (T/xT).
\]
Arguing similarly as above for the ideal $T+(x)=xR+T$, we have $\lind_R (T+(x))=\lind_R T/xT$. Hence $\lind_R I=\lind_R (T/xT)=\lind_R (T+(x))$, as claimed.

\textsf{Step 2:} Note that $x$ is $T$-regular, since it is $\gr_\mm T$-regular. Let $K(x;T)$ denote the Koszul complex $0\to T(-1) \xrightarrow{\cdot x} T \to 0$, then $K(x;T)$ is quasi-isomorphic to $T/xT$. Hence using (the graded analogue of) a result of Iyengar and R\"omer \cite[Remark 2.12]{IyR}, we obtain
\[
\lind_R T \le \lind_R K(x;T) =\lind_R (T/xT)\le \lind_R T+1.
\]
Hence we get the desired chain
$$
\lind_R T \le \lind_R I = \lind_R (T/xT) \le \lind_R T+1.
$$

\textsf{Step 3:} For the assertion on Betti splitting, note that $xI'\cap T=xT$. Since $x$ is $T$-regular, then the morphism $\Tor^R_i(k,xT)\longrightarrow \Tor^R_i(k,T)$ is the multiplication by $x$ of $\Tor^R_i(k,T)$, which is trivial. Since $xT\subseteq \mm xI'$ and $xI'$ is Koszul, the map $\Tor^R_i(k,xT)\longrightarrow \Tor^R_i(k,xI')$ is also trivial thanks to \cite[Lemma 4.10(b1)]{NgV2}. Hence by Lemma \ref{lem_criterion_Bettisplit}, the decomposition $I=xI'+T$ is a Betti splitting.

\textsf{Step 4:} As shown above, $\lind_R T\le \lind_R I$, hence if $I$ is Koszul then so is $T$. Conversely, assume that $T$ is Koszul. Now $x$ is $\gr_\mm T$-regular, so by (the graded analogue of) \cite[Theorem 2.13(a)]{IyR}, we deduce that $T/xT$ is also Koszul. It remains to use the equality $\lind_R I=\lind_R T/xT$.
The proof is concluded.
\end{proof}
\begin{exm}
\label{ex_grmT-reg_necessary}
The following example shows that the condition $x$ is $\gr_{\mm} T$-regular in Theorem \ref{thm_ysplit} is critical, even when the base ring is regular.

Let $R=k[a,b,c,d]$, $T=(a,c^2)(b,d^2)=(ab,ad^2,bc^2,c^2d^2)$. Let $x=a-b$,  $I=(x)+T$, $\mm=R_+$. We claim that:
\begin{enumerate}[\quad \rm(i)]
\item $x$ is $(R/T)$-regular but not $\gr_\mm T$-regular,
\item $T$ is Koszul but $I$ is not.
\end{enumerate}

(i): We observe that $T:x=T$ because $T=(a,c^2)\cap (b,d^2)$. We also have
\[
c^2d^2x=c^2(ad^2)-d^2(bc^2) \in \mm^2 T.
\]
Hence $c^2d^2 \in (\mm^2T:x)\setminus (\mm T)$, thus $x$ is not $\gr_{\mm}T$-regular.

(ii):  Write $T=aJ+U$ where $J=(b,d^2),U=c^2(b,d^2)$. Then $J$ is Koszul, $U\subseteq \mm J$ and $U\cong J(-2)$ is Koszul. Applying Corollary \ref{cor_Koszulysplit}, $T$ is also Koszul.

We have
\begin{align*}
I= T+(x) &=(x)+(a^2,ac^2,ad^2,c^2d^2).
\end{align*}
Denote $L=(a^2,ac^2,ad^2,c^2d^2)\subseteq S=k[a,c,d]$. Note that $R=S[x]$, so by Corollary \ref{cor_Koszulysplit} and Lemma \ref{lem_basechange_ld}, $\lind_R I=\lind_R (LR+(x))=\lind_S L$.

Assume that $I$ is Koszul, then so is $L$. Denote by $L_{\le s}$ the ideal generated by homogeneous elements of degree at most $s$ of $L$. Then by \cite[Lemma 8.2.11]{HH2}, we also have $L_{\le 3}=(a^2,ac^2,ad^2) \cong (a,c^2,d^2)(-1)$ is Koszul. In particular, by Lemma \ref{lem_Koszul_gendeg}, $\reg L_{\le 3}=3$.

But then $\reg (a,c^2,d^2)=2$! This contradiction confirms that $I$ is not Koszul.
\end{exm}
\begin{rem}
Example \ref{ex_grmT-reg_necessary} also shows that even if $T$ is a Koszul ideal in a polynomial ring $R$, and $x$ is a regular linear form modulo $T$, the ideal $T+(x)$ need not be Koszul. 

Nevertheless, it is not hard to see that this is true if moreover $T$ has a linear resolution. Indeed, in this case $T\cong \gr_{\mm}T$ as $R$-modules, so $x$ is $\gr_{\mm}T$-regular. Applying Theorem \ref{thm_ysplit}, we get $\lind_R (T+(x))=\lind_R T=0$.
\end{rem}
The next consequence of Theorem \ref{thm_ysplit} generalizes \cite[Lemma 8.2]{NgV2}.
\begin{cor}
\label{cor_Koszulysplit}
Let $(R,\mm)$ be a polynomial ring over $k$. Let $x$ be a non-zero linear form, $I', T$ be non-trivial homogeneous ideals of $R$ such that the following conditions are satisfied:
\begin{enumerate}[\quad \rm(i)]
 \item $I'$ is Koszul,
 \item $T\subseteq \mm I'$,
 \item there exists a polynomial subring $S$ of $R$ such that $R=S[x]$ and $T$ is generated by elements in $S$.
\end{enumerate}
Denote $I=xI'+T$. Then the decomposition $I=xI'+T$ is a Betti splitting and $\lind_R I=\lind_R T$.
\end{cor}
\begin{proof}
First we verify that $x, I',$ and $T$ satisfy the hypotheses of Theorem \ref{thm_ysplit}. Note that condition (iii) ensures that $x$ is $(R/T)$-regular. Hence it remains to check that $x$ is $\gr_\mm T$-regular. By the proof of Lemma \ref{lem_regular_lin_form}, we only need to show that for all $s\ge 1$, 
\[
\mm^s T:x \subseteq \mm^{s-1}T.
\]
Take $a\in \mm^sT:x$.

By change of coordinates, we can assume that $x$ is one of the variables. Let $\nn$ be the graded maximal ideal of $S$ extended to $R$. Then $\mm^s=((x)+\nn)^s=x\mm^{s-1}+\nn^s$, therefore
\[
xa \in \mm^sT=x\mm^{s-1} T+\nn^sT.
\]
So for some $b\in \mm^{s-1}T$, $x(a-b) \in \nn^sT$, namely
\[
a-b\in \nn^sT:x=\nn^sT.
\]
Therefore $a\in \mm^{s-1}T+\nn^sT=\mm^{s-1}T$, as claimed.

That $I=xI'+T$ is a Betti splitting follows from Theorem \ref{thm_ysplit}. 

Regarding $T$ as an ideal of $S$, by Theorem \ref{thm_ysplit}, we also have
\[
\lind_R I = \lind_R T/xT=\lind_R \left(T\otimes_k \frac{k[x]}{(x)}\right)=\lind_S T,
\]
where the last equality holds because of \cite[Lemma 2.3]{NgV3}. Hence $\lind_R I=\lind_R T$, as desired.
\end{proof}

The main result of this section is as follows.  

\begin{thm}
\label{thm_Koszul_symbolicpower_coverideals} 
Let $G$ be the graph obtained by adding to each vertex of a graph $H$ at least one pendant. Then all the symbolic powers of the cover ideal $J(G)$ of $G$ are Koszul.
\end{thm}

First, we need an auxiliary lemma. If $m = x_1^{\alpha_1}\cdots x_r^{\alpha_r}$ is a monomial of $R$, its support is defined by $\supp(m) =\{x_i \mid \alpha_i\ne 0\}$. For a set $B$ of monomials in $R$, set $\supp B = \bigcup_{m\in B} \supp(m)$. 

\begin{lem}\label{Koszul-Products}  Let $I$ be a momomial ideal of $S=k[x_1,\ldots,x_s]$.  Let $1\le t \le s$ be an integer. Assume that for every monomial $m\in S$ with $\supp(m) \subseteq \{x_1,\ldots,x_t\}$, the ideal $(m) \cap I$ is Koszul. Denote $R=S[y,z]$. Let $m_1=y^{\alpha}fg$, where $\alpha\ge 1$ is an integer and $f,g$ are monomials of $S$ satisfying the following conditions:
\begin{enumerate}[\quad \rm(i)]
\item  $\supp(g)\subseteq \{x_1,\ldots,x_t\}$,
\item $\supp f \cap \big(\supp \Gcc(I) \cup \{x_1,\ldots,x_t\} \big)=\emptyset$.
\end{enumerate}
Then  for all monomials $m\in R$ with $\supp(m)\subseteq \{x_1,\ldots,x_t\}$ and all $p, q\geqslant 0$, the ideal $(z,m_1)^p\cap (mz^q) \cap I$ is Koszul.
\end{lem}

\begin{proof} We prove by induction on $p+q$. If $p+q= 0$, then $p=q=0$. In this case, the conclusion holds true by the assumption.

Assume that $p+q \geqslant 1$. If $p \leqslant q$, then  we have 
$$(z,m_1)^p \cap (m z^q) \cap I = (mz^q) \cap I.$$
Since $(mz^q) \cap I = z^q((m)\cap I)$, we have $ (mz^q) \cap I$ is Koszul.

Assume that $p > q$. Consider two cases.

\textsf{Case 1}: $q = 0$. We have
\begin{align*}
(z,m_1)^p &\cap (m) \cap I =  (z,m_1)^p \cap ((m) \cap I)\\
&=(z(z,m_1)^{p-1}+(m_1^p)) \cap ((m) \cap I) \\
&= zJ +L
\end{align*}
where $J = (z,m_1)^{p-1}\cap (m) \cap I$ and $L = (m_1^p)\cap  (m) \cap I$.

Observe that $J$ is Koszul by the induction hypothesis. From the assumptions, $\supp(y^{\alpha}f)\cap \left(\supp(m)\cup \supp \Gcc(I)\right)=\emptyset$, so 
\begin{align*}
L&= (y^{p\alpha}f^pg^p)\cap  (m) \cap I=y^{p\alpha}f^p((g^p)\cap (m) \cap I)\\
 &=y^{p\alpha}f^p(\lcm(g^p,m) \cap I).
\end{align*}
Since $\supp \lcm(g^p,m) \subseteq \{x_1,\ldots,x_t\}$, the assumptions yields that $\lcm(g^p,m) \cap I$ is Koszul. Therefore $L$ is Koszul.

The above arguments also give 
\begin{align*}
L \subseteq y^{\alpha}\left(y^{(p-1)\alpha}f^pg^p \cap (m) \cap I\right) \subseteq y\left((m_1^{p-1})\cap (m) \cap I \right) \subseteq y J.
\end{align*} 
Thus $(z,m_1)^p \cap (m) \cap I$  is Koszul by Corollary $\ref{cor_Koszulysplit}$.

\textsf{Case 2}: $q\geqslant 1$. Then
\[
(z,m_1)^p \cap (m z^q) \cap I = z^q ((z,m_1)^{p-q} \cap (m) \cap I),
\]
which is Koszul by the induction hypothesis. The proof is complete.
\end{proof}
Now we present the
\begin{proof}[Proof of Theorem \ref{thm_Koszul_symbolicpower_coverideals}]
Assume that $V(H) = \{x_1,\ldots,x_d\}$. Let $R=k[x: x\in V(G)]$. In order to prove the theorem we prove the stronger statement that $(m) \cap J(G)^{(n)}$ is Koszul for every monomial $m\in R$ with $\supp(m)\subseteq \{x_1,\ldots,x_d\}$. Choosing $m=1$, we get the desired conclusion.  

Induct on $d$. 

\textsf{Step 1}: If $d=1$, then $G$ is a star with the edge set $E(G) = \{x_1y_1,\ldots,x_1y_e\}$, where $e\ge 1$. In this case, $J(G) = (x_1,y_1\cdots y_e)$ and $J(G)^{(n)} = (x_1,y_1\cdots y_e)^n$. Assume that $m = x_1^p$. Since
$$
(m) \cap J(G)^{(n)} = (x_1^p) \cap (x_1,y_1\cdots y_e)^n =x_1^p (x_1,y_1\cdots y_e)^{\max\{0,n-p\}},
$$
it suffices to prove that $(x_1,y_1\cdots y_e)^n$ is Koszul for all $n\ge 0$.

If $n=0$, this is clear. Assume that $n\ge 1$ and the statement holds for $n-1$. We write $x=x_1, h=y_1\cdots y_e$. Then
\[
(x_1,y_1\cdots y_e)^n = (x,h)^n=x(x,h)^{n-1}+(h^n).
\]
By the induction hypothesis, $(x,h)^{n-1}$ is Koszul. Applying Corollary \ref{cor_Koszulysplit},
\[
\lind_R (x,h)^n=\lind_R (h^n)=0.
\]

\textsf{Step 2}: Assume that $d\geqslant 2$. Let $y_1,\ldots,y_e$ be the vertices of the pendants of $G$ which are adjacent to $x_d$, where $e\ge 1$. Let $H' = H\setminus \{x_d\}$ and $G' = G\setminus\{x_d,y_1,\ldots,y_e\}$. Then 
$V(H') = \{x_1,\ldots,x_{d-1}\}$ and $G'$ is obtained by adding to each vertex of $H'$ at least one pendant. 

Let $S$ be the polynomial ring with variables being the vertices of $G\setminus \{x_d,y_1\}$. Denote $I=J(G')^{(n)}$.
By the induction hypothesis and Lemma \ref{lem_basechange_ld}, $(m')\cap I$ is Koszul for every monomial $m'\in S$ with $\supp(m') \subseteq\{x_1,\ldots,x_{d-1}\}$.

Denote $y=y_1,z=x_d$, then $R=S[y,z]$.

Let $m_1 =\prod_{x\in N_G(z)}x$, $f=y_2\cdots y_e$, $g=\prod_{x\in N_H(z)}x$.  Then 
\begin{enumerate}[\quad \rm(i)]
\item $m_1=yfg$,
\item $\supp(g) \subseteq \{x_1,\ldots,x_{d-1}\}$,
\item $\supp(f) \cap \big(\supp \Gcc(I) \cup \{x_1,\ldots,x_{d-1}\} \big)=\emptyset$.
\end{enumerate}
Moreover by Fact \ref{fact_symbolic_power_monomial},
\begin{align*}
J(G)^{(n)} &= (z,y_1)^n \cap \cdots \cap (z,y_e)^n \cap \bigcap_{x\in N_H(z)} (z,x)^n \cap  J(G')^{(n)}\\
           &=(z,y_1\cdots y_e g)^n \cap I =(z,m_1)^n \cap I.
\end{align*}
The second equality holds by observing that $(z,y_1\cdots y_eg)$ is a complete intersection, or by direct inspection.

Take any monomial $m\in R$ with $\supp (m) \subseteq \{x_1,\ldots,x_{d-1},z\}$. We can write $m = m'z^p$ where $\supp(m') \subseteq \{x_1,\ldots,x_{d-1}\}$. Hence
$$
(m) \cap J(G)^{(n)} = (z,m_1)^n \cap (m)  \cap I=(z,m_1)^n \cap (m'z^p)  \cap I.
$$  
By Lemma $\ref{Koszul-Products}$, the last ideal is Koszul. This finishes the induction on $d$ and the proof.
\end{proof}

The {\it corona} $\cg(G)$ of a graph $G$ is the graph obtained from $G$ by adding a pendant at each vertex of $G$. More generally, the generalized corona $\cg(G, s)$ is the graph obtained from $G$ by adding $s\ge 1$ pendant edges to each vertex of $G$ (see Figure \ref{fig_corona}).

By Alexander duality \cite[Chapter 8]{HH2}, we know that the edge ideal $I(G)$ is Koszul (having a linear resolution) if and only if $J(G)$ is sequentially Cohen-Macaulay (respectively, Cohen-Macaulay). Combining this with work of Villarreal \cite[Section 4]{Vi1}, Francisco and H\`a \cite[Corollary 3.6]{FH}, we know that $J(\cg(G))$ has a linear resolution. We generalize this for all symbolic powers of $J(\cg(G))$ as follows.

\begin{cor} 
\label{cor_linear_res_symbolpow}
Let $G$ be a simple graph. Then all the symbolic powers of the cover ideal $J(\cg(G))$ have linear resolutions.
\end{cor}

We introduce some more notation. Let $V(G)=\{x_1,\ldots,x_r\}$, where it is harmless to assume that $r\ge 1$. Let $y_1,\ldots,y_r$ be the new vertices in $V(\cg(G))$, where $y_i$ is only adjacent to $x_i$ for all $i=1,\ldots,r$. 

\begin{convn}
\label{convn_coordinates}
We denote the coordinates of the ambient $\R^{2r}$ containing $\smp(J(\cg(G)))$ by $x_1,\ldots,x_r,y_1,\ldots,y_r$ instead of $x_1,\ldots,x_r,x_{r+1},\ldots,x_{2r}$, thus $y_i=x_{r+i}$ for $i=1,\ldots,r$.
\end{convn}
 The proof of Corollary \ref{cor_linear_res_symbolpow} depends on the following lemma (where Convention \ref{convn_coordinates} is in force).

\begin{lem}
\label{lem_vertex_corG}
Denote $J=J(\cg(G))$. Then for any vertex $\alb \in \R^{2r}$ of $\smp(J)$, up to a relabeling of the variables, there exist integers $0\le p\le q \le r$ such that $\alb$ is a solution of the following system:
\[
\begin{cases}
x_1=\cdots=x_p=y_{p+1}=\cdots=y_q=0,\\
y_1=\cdots=y_p=x_{p+1}=\cdots=x_q=1,\\
x_j=y_j=1/2, \quad \text{if $q+1\le j\le r$}.
\end{cases}
\]
In particular, $|\alb|=r$.
\end{lem}
\begin{proof}
For the first assertion, note that by Lemma \ref{lem_vertex_smpJG}, $\alb_i \in \{0,1,1/2\}$ for all $i=1,\ldots,2r$. Denote $S_0=\{x_i: \alb_i=0\}$, $S_1=\{x_i: \alb_i=1\}$, $S_{1/2}=\{x_i: \alb_i=1/2\}$. By Lemma \ref{lem_vertex_smpJG}, we also have $S_0$ is an independent set of $\cg(G)$.

Without loss of generality, we can assume that $S_0=\{x_1,\ldots,x_p, y_{p+1},\ldots,y_q\}$ for some $0\le p\le q \le r$ (recall Convention \ref{convn_coordinates}). We have to show that $S_1=\{y_1,\ldots,y_q,x_{p+1},\ldots,x_q\}$.

By Lemma \ref{lem_vertex_smpJG}, $\{y_1,\ldots,y_q,x_{p+1},\ldots,x_q\} \subseteq N(S_0)=S_1$. Clearly $y_{q+1},\ldots,y_r \notin S_1$ since none of them belongs to $N(S_0)$. Hence it remains to show that $x_i\notin S_1$ for $q+1\le i\le r$.

By the definition of $S_0$, $y_i\notin S_0$. Now $y_i$ is a leaf of $\cg(G)$ and $N(y_i)=\{x_i\}$, so by Lemma \ref{lem_vertex_smpJG}, $x_i\notin S_1$, as desired. 

The second assertion now follows from accounting. The proof is concluded.
\end{proof}

\begin{proof}[Proof of Corollary \ref{cor_linear_res_symbolpow}] It is harmless to assume that $r = |V(G)|\ge 1$, as mentioned above. By Theorem \ref{thm_Koszul_symbolicpower_coverideals} it suffices to show that $J^{(n)}=J(\cg(G))^{(n)}$ generated by monomials of degree $rn$. 

\textsf{Step 1:} Take any vertex $\v \in \R^{2r}$ of $\smp(J)$. By Lemma \ref{lem_vertex_corG}, it follows that $|\v|=r$; in particular $\delta(J) = r$.

\textsf{Step 2:} Let $x^{\alb}$ be a minimal generator of $J^{(n)}$. Since $\alb \in \smp_n(J)$, we get $\dfrac{1}{n}\alb \in \smp(J)$. Together with Step 1, it follows that
$$
\frac{1}{n} |\alb| \geqslant \min\{|\v| \mid \v \text{ is a vertex of } \smp(J)\}= r,
$$
namely  $|\alb| \geqslant nr$. 

On the other hand, by Lemma \ref{lem_upper_bound_omega_via_delta}, $$|\alb| \leqslant \omega(J^{(n)}) \leqslant \delta(J)n = rn.$$
Thus $|\alb| = nr$, as require.
\end{proof}

\begin{rem}
A graph $G$ which contains no induced cycle of length at least $4$ is called a \emph{chordal} graph. We say that $G$ is a \emph{star graph based on a complete graph $K_m$} if $G$ is connected and $V(G)=\{1,\ldots,m,m+1,\ldots,m+g\}$ for some $g\ge 0$ such that: 
\begin{enumerate}
\item the complete graph on $\{1,\ldots,m\}$ is a subgraph of $G$, and,
\item there is no edge in $G$ connecting $i$ and $j$ for all $m+1 \le i<j \le m+g$.
\end{enumerate}
Any star graph based on a complete graph is chordal.

Let $G$ be a chordal graph. Francisco and Van Tuyl \cite[Proof of Theorem 3.2]{FV} showed that for such a $G$, $J(G)$ is Koszul\footnote{This result can be proved quickly using Corollary \ref{cor_Koszulysplit}.}. In \cite{HHO}, Herzog, Hibi and Ohsugi conjectured that all the powers of $J(G)$ are Koszul. Furthermore, in \emph{ibid.}, Theorem 3.3, they confirmed this in the case $G$ is a star graph based on a complete graph $K_m$. Hence it is natural to ask: If $G$ is star graph based on a complete graph, is it true that $J(G)^{(n)}$ Koszul for all $n\geqslant 1$?

The answer is ``No!'' Here is a counterexample. Consider the graph $G_2$ in Figure \ref{fig_diamond}. It is the complete graph on the vertices $\{a,b,c,d\}$ with one edge removed. The corresponding cover ideal is
\[
J=J(G_2)=(bc,abd,acd).
\]
Since $G_2$ is a star graph based on $K_2$, $J$ and all of its ordinary powers are Koszul by \cite[Theorem 3.3]{HHO}. But $J^{(n)}$ is not Koszul for all $n\ge 2$ by \cite[Page 186]{CE}.
\end{rem}

It is natural to ask
\begin{quest}
Classify all star graphs based on a complete graph $G$ such that all the symbolic powers of $J(G)$ are Koszul.
\end{quest}

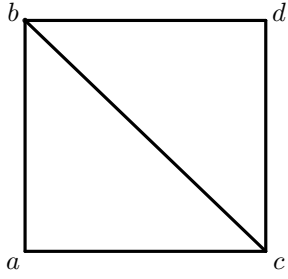
\begin{figure}
\centering
\begin{tikzpicture}[scale=0.8,line cap=round,line join=round,>=triangle 45,x=1.0cm,y=1.0cm]
\clip(3.5,-1.32) rectangle (10.28,4.4);
\draw [line width=1.2pt] (4.98,3.34)-- (8.98,3.34);
\draw [line width=1.2pt] (4.98,3.34)-- (4.98,-0.5);
\draw [line width=1.2pt] (4.98,3.34)-- (8.98,-0.5);
\draw [line width=1.2pt] (8.98,3.34)-- (8.98,-0.5);
\draw [line width=1.2pt] (4.98,-0.5)-- (8.98,-0.5);
\draw [color=black] (4.98,3.34) circle (1.0pt);
\draw[color=black] (4.78,-0.7) node {$a$};
\draw[color=black] (4.78,3.48) node {$b$};
\draw[color=black] (9.20,-0.7) node {$c$};
\draw[color=black] (9.20,3.48) node {$d$};
\end{tikzpicture}
\caption{The graph $G_2$}
\label{fig_diamond}
\end{figure}

Observe that a subset $\tau \subseteq V(G)$ is a minimal vertex cover of $G$ if and only if $V(G)\setminus \tau$ is a maximal independent set of $G$.

Let $K_m$ be a complete graph with $m$ vertices and $G = \cg(K_m,s)$ where $m\geqslant 3$ and $s\geqslant 2$. In the rest of the paper we show that both $\omega(J(G)^{(n)})$ and $\reg (J(G)^{(n)})$ are not necessarily asymptotic linear functions in $n$. 
\begin{lem}
\label{lem_deg_delta_corKms} For any $m\geqslant 3$ and $s\geqslant 2$, we have:
\begin{enumerate}[\quad \rm(1)]
\item $\omega(J(\cg(K_m,s))) = m+s-1$.
\item $\delta(J(\cg(K_m,s)))=\frac{1}{2}m(s+1)$.
\end{enumerate}
\end{lem}
\begin{proof} Let $G = \cg(K_m,s)$. Then $G$ has $r=m(s+1)$ vertices and $ms$ leaves.

(1) Let $S$ be a maximal independent set of $G$. Then either $S$ is the set of leaves of $G$, or $S$ consists of a vertex of $K_m$ and  $(m-1)s$ leaves which are incident with the remaining vertices of $K_m$. Thus, the cover number of $G$ is 
$$m(s+1) - (1+(m-1)s)=m+s-1,$$
and thus $\omega(J(G)) = m+s-1$.

(2)  Since $|V(G)|=m(s+1),$ by Theorem \ref{thm_delta_JG} we only need to prove the following: Let $S$ be an independent set of $G$ such that $G\setminus N[S]$ has no bipartite components. Then $|N(S)| \le |S|$, with equality happens when $S=\emptyset$. 

We consider three cases:

\textbf{Case 1}: $S = \emptyset$. Then $N(S)=\emptyset$ and $|N(S)| - |S| = 0$.

\textbf{Case 2}: $S$ contains a vertex of $K_m$, say $v$. Then $G\setminus N[S]$ is either empty or totally disconnected, in which case it is bipartite. Since $G\setminus N[S]$ has no bipartite component, the first alternative happens. It follows that $S$ consists of $v$ and all the leaves not adjacent to it. Thus, $|S| = 1 + (m-1)s \ge |N(S)| = s + m -1$, since
\[
1+(m-1)s-(s+m-1)=(m-2)(s-1)>0.
\]

\textbf{Case 3}: $S$ contains only leaves of $G$. Let $x_1,\ldots,x_m$ be vertices of $K_m$. Then $N(S)$ consists only of vertices of $K_m$, say $x_1,\ldots,x_t$ for $1\le t\le m$. Each $x_i$ requires at least a leaf adjacent to it, so clearly $|S|\ge t = |N(S)|$. 

The proof is concluded.
\end{proof}

Finally, we present a family of counterexamples to Question \ref{quest_2_linear_behavior}.
\begin{thm} 
\label{thm_non-linear}
Let $G=\cg(K_m,s)$ where $m\geqslant 3$ and $s\geqslant 2$. Let $J=J(G)$ be its cover ideal. Then for all $n\geqslant 0$,
\begin{enumerate}[\quad \rm(1)]
\item $\reg J^{(2n)}=\omega(J^{(2n)}) = m(s+1)n$;
\item $\reg (J^{(2n+1)})=\omega(J^{(2n+1)}) = m(s+1)n + m+s-1$.
\end{enumerate}
In particular, for all $n$,
$$
\reg (J^{(n)})=\omega(J^{(n)})=(m+s-1)n + (m-2)(s-1) \left\lfloor \frac{n}{2} \right \rfloor,
$$ 
which is not an eventually linear function of $n$.
\end{thm}
\begin{proof} By Theorem \ref{thm_Koszul_symbolicpower_coverideals}, $J^{(n)}$ is Koszul for all $n\ge 1$. Hence by Lemma \ref{lem_Koszul_gendeg}, $\reg (J^{(n)})=\omega(J^{(n)})$ for all $n$. 

Note that by Lemma \ref{lem_deg_delta_corKms}, $\delta(J)=\delta(J(G))=m(s+1)/2$, namely half the number of vertices of $G$. Hence by Theorem \ref{thm_gen_degree_symbolicpowers}, for all $n\geqslant 0$
\begin{align*}
\omega(J^{(2n)}) &=2n\delta(J),\\
\omega(J^{(2n+1)}) &= 2n\delta(J) + \omega(J).
\end{align*}
From Lemma \ref{lem_deg_delta_corKms}(1), $\omega(J)=m+s-1$, so the desired formulas follow.
\end{proof}

\section*{Acknowledgments}  
L.X. Dung, T.T. Hien and T.N. Trung are partially supported by NAFOSTED (Vietnam) under the grant number 101.04-2018.307. H.D. Nguyen is partially supported by International Centre for Research and Postgraduate Training in Mathematics (ICRTM) under grant number ICRTM01$\_$2020.05. H.D. Nguyen and T.N. Trung  are also grateful to the support of Project CT 0000.03/19-21 of the Vietnam Academy of Science and Technology. Part of this work was done during our stay at the Vietnam Institute for Advanced Study in Mathematics.

Finally, the authors would like to thank the anonymous referee for useful comments which have helped them to improve the quality 
of this paper.

\end{document}